\documentclass[bibnum]{preprint}

\title{Uniqueness and numerical resolution via iterated sensitivity equation of an inverse electromagnetic coefficient problem with partial boundary data}

\addauthor{Jérémy}{Heleine}{jeremy.heleine@math.univ-toulouse.fr}{Institut de Mathématiques de Toulouse CNRS UMR 5219\\Université Paul Sabatier F-31062 Toulouse Cedex 9, France}

\date{October 1, 2024}

\abstract{
	In this paper we study an inverse boundary value problem for Maxwell's equations. The goal is to reconstruct perturbations in the refractive index of the medium inside an object from the knowledge of the tangential trace of an electric field on a part of the boundary of the domain. We first provide a uniqueness result for this inverse problem. Then, we propose a new approach to reconstruct numerically the perturbations. This complete procedure is based on the minimization of a cost functional involving an iterated sensitivity equation.
}

\begin{document}

	\section{Introduction}

	The study of microwave imaging is of great interest, with potential medical or industrial applications. The idea is to take benefits of the fact that the interaction between electromagnetic fields and a material can be described from the dielectric properties of the latter. For example, this modality is in study for medical applications like the diagnostic of strokes \cite{SemenovSeiserStoegmannAuff14, Tournieretal17}. Indeed, a stroke affects tissues of the brain, resulting in variations of their dielectric properties. The aim is then to recover these variations in order to characterize the stroke and apply the right treatment to the patient. In the same way, microwave imaging is a promising alternative to mammography \cite{KwonLee16}.

	Mathematically speaking, microwave imaging defines an inverse problem. Let $\Omega$ be a bounded and simply connected domain of $\R^p$, $p \in \{2,3\}$, with boundary $\Gamma \coloneqq \partial\Omega$ and unit outward normal denoted by $\bfn$. The magnetic permeability of the medium in $\Omega$ is assumed to be the same as in the vacuum, $\mu_0$. The electric permittivity and conductivity in the medium are two functions of the space variable $\bfx$, respectively denoted by $\eps$ and $\sigma$. We then define the complex refractive index by
	\[\kappa \colon \Omega \ni \bfx \mapsto \frac{1}{\eps_0}\left(\eps(\bfx) + i\frac{\sigma(\bfx)}{\omega}\right) \in \C,\]
	where $\eps_0$ is the electric permittivity in vacuum and $\omega > 0$ is the (fixed) wave frequency. The electric field $\bfE$ in $\Omega$ then satisfies the time-harmonic Maxwell's equation
	\begin{equation}
		\label{eq:maxwell}
		\curl\curl\bfE - k^2\kappa\bfE = 0,
	\end{equation}
	with $k \coloneqq \omega\sqrt{\mu_0\eps_0}$ being the wave number.

	We are interested in the inverse boundary value problem consisting in recovering the refractive index in the whole domain $\Omega$ from the knowledge of partial surface measurements at fixed frequency. The part of the boundary $\Gamma_0 \subset \Gamma$ where measurements will be assumed to be known is called the accessible part. We assume that $\operatorname{meas}(\Gamma_0) > 0$ and we denote by $\Gamma_1 \coloneqq \Gamma \setminus \overline{\Gamma_0}$ the nonaccessible part. In this paper, we study the question of the uniqueness for such a problem. We then provide an algorithm to numerically solve this inverse problem.

	We assume here that the unknown function $\kappa$ is a perturbation of a background refractive index, denoted by $\kappa_0$ and assumed to be known in $\Omega$. Moreover, we assume that the perturbations are not glued to the boundary. In other words, there exists a tubular neighborhood $\mcV$ of $\Gamma$ in $\Omega$, that is:
	\[\exists r_\mcV > 0 / \forall \bfx \in \Gamma, B(\bfx, r_\mcV) \cap \Omega \subset \mcV,\]
	where $B(\bfx, r_\mcV)$ is the open ball centered at $\bfx$ of radius $r_\mcV$, such that $\supp(\kappa - \kappa_0)$ is contained in $\Omega \setminus \bar{\mcV}$. This assumption allows us to write $\restriction{\kappa}{\mcV} = \restriction{\kappa_0}{\mcV}$: the refractive index is known in $\mcV$. An example of such configuration is provided in \autoref{fig:configuration}.

	Under this assumption, this problem shares similarities with the reconstruction of obstacles from far-field measurements, in the sense that both aim to retrieve unknowns properties from field data. The key difference lies in the data: we are using here near-field measurements. This leads to approaches that are different from the ones involving far-field measurements. Examples of such approaches, in the case of Maxwell's equations, can be found in \cite{HeLiLiuWang24, ChenDengGaoLiLiu24, LiuRondiXiao19, HaddarMonk02} and references there-in.

	Assuming that the refractive index is known in a tubular neighborhood $\mcV$ of $\Gamma$ has two benefits. From a theoretical point of view, it allows us to state a new uniqueness result without having to fix some specific geometric nor physical conditions on the nonaccessible part of the boundary. This will be done in \autoref{sec:uniqueness}. From a practical point of view, this assumption also allows us to have a well-defined system for the data completion problem. In \autoref{sec:transmission}, we then propose a way to retrieve total data from partial ones: starting with data known on the accessible part of the boundary, we obtain data defined on a complete boundary. \autoref{sec:sensitivity} is devoted to an application mapping amplitudes of perturbations to corresponding electric fields. We show that this application can be differentiated indefinitely, the value of its derivatives being given by the so called iterated sensitivity equation. This yields an original approach to regularize a classical cost function in \autoref{sec:cost}. In \autoref{sec:numerics}, we describe a complete procedure to reconstruct perturbations (both their supports and amplitudes) from the knowledge of the electric field on the accesssible part. This procedure is then tested against different 2D and 3D configurations, first to retrieve perturbations of simple shape, and then to study how the algorithm behaves with more complex inhomogeneities.

	\section{Uniqueness result}
	\label{sec:uniqueness}

	Let us introduce the vector space
	\[H(\curl) \coloneqq \set{\bfu \in L^2(\Omega)^3}{\curl\bfu \in L^2(\Omega)^3}.\]
	For any vector field $\bfu \in H(\curl)$, the tangential trace is defined by continuous extension of the mapping $\gamma_t(\bfu) = \restriction{\bfu}{\Gamma} \times \bfn$ (see for example \cite{Monk03}). We introduce the trace space
	\[Y(\Gamma) \coloneqq \set{\bff \in H^{-1/2}(\Gamma)^3}{\exists \bfu \in H(\curl) / \gamma_t(\bfu) = \bff},\]
	and its restriction in the distributional sense to $\Gamma_0$
	\[Y(\Gamma_0) \coloneqq \set{\restriction{\bff}{\Gamma_0}}{\bff \in Y(\Gamma)}.\]

	The inverse boundary value problem we are interested in is to determine the electric permittivity $\eps$ and conductivity $\sigma$ from boundary measurements taken on the accessible part $\Gamma_0$, obtained from a given boundary source term, at a fixed frequency $\omega$. These measurements are modeled by a Cauchy data set $C(\eps, \sigma; \Gamma_0)$, as defined in the following.

	\begin{definition}
		\label{def:admissible}
		The pair of coefficients $\eps$ and $\sigma$ is said admissible if both are in $\mcC^1(\bar{\Omega})$, with $\eps \geq \tilde{\eps}$ and $\sigma \geq \tilde{\sigma}$ almost everywhere in $\Omega$, for some constants $\tilde{\eps} > 0$ and $\tilde{\sigma} > 0$.
	\end{definition}

	\begin{definition}
		\label{def:cauchy_data}
		For a pair of admissible coefficients $(\eps,\sigma)$ defined in $\Omega$ as in \autoref{def:admissible}, the corresponding Cauchy data set $C(\eps, \sigma; \Gamma_0)$ at a fixed frequency $\omega > 0$ consists of pairs $(\bff, \bfg) \in Y(\Gamma_0) \times Y(\Gamma_0)'$ such that there exists a field $\bfE \in H(\curl)$ satisfying \eqref{eq:maxwell} for $\kappa \coloneqq (\eps + i\sigma / \omega) / \eps_0$ with boundary conditions $\restriction{\bfE}{\Gamma_0} \times \bfn = \bff$ and $\restriction{\curl\bfE}{\Gamma_0} \times \bfn = \bfg$.
	\end{definition}

	The uniqueness question for our inverse problem reads as follows. Given a frequency $\omega > 0$ and two pairs of admissible coefficients $(\eps_j,\sigma_j)$, $j \in \{1,2\}$, does $C(\eps_1, \sigma_1; \Gamma_0) = C(\eps_2, \sigma_2; \Gamma_0)$ imply $\eps_1 = \eps_2$ and $\sigma_1 = \sigma_2$ in $\Omega$? A uniqueness result for partial data is stated in \cite{Caro11}. However, geometrical conditions are imposed on the nonaccessible part $\Gamma_1$, which is supposed to be part of either a plane or a sphere. In \cite{Brownetal16}, these geometrical conditions are relaxed, but are replaced by a perfect conducting boundary condition: $\restriction{\bfE}{\Gamma_1} \times \bfn = 0$. Both hypotheses can be restrictive in applications. The result stated in \autoref{thm:uniqueness} can be seen as an improvement of these results, as neither geometrical nor boundary conditions are imposed on the nonaccessible part. Although simple, the idea behind \autoref{thm:uniqueness} is new, and uses results from \cite{Brownetal16} and \cite{CaroZhou14}.

	\begin{theorem}
		\label{thm:uniqueness}
		Let $\omega > 0$. Assume that $(\eps_j, \sigma_j)$, $j \in \{1,2\}$, are two pairs of admissible coefficients (see \autoref{def:admissible}) such that $\eps_1 = \eps_2$ and $\sigma_1 = \sigma_2$ in $\bar{\mcV}$ where $\mcV$ is a tubular neighborhood of $\Gamma$. Then $C(\eps_1, \sigma_1; \Gamma_0) = C(\eps_2, \sigma_2; \Gamma_0)$ implies $\eps_1 = \eps_2$ and $\sigma_1 = \sigma_2$ in $\Omega$.
	\end{theorem}

	\begin{proof}
		We begin to prove that, under these hypotheses, Cauchy data sets coincide not only on $\Gamma_0$, but on the whole boundary $\Gamma$. To this end, consider a couple $(\bff, \bfg) \in C(\eps_1, \sigma_1; \Gamma)$. Then, there exists $\bfE_1 \in H(\curl)$ satisfying
		\[
			\left\{
				\begin{aligned}
					\curl\curl\bfE_1 - k^2\kappa_1\bfE_1 &= 0, && \text{in } \Omega, \\
					\bfE_1 \times \bfn &= \bff, && \text{on } \Gamma, \\
					\curl\bfE_1 \times \bfn &= \bfg, && \text{on } \Gamma,
				\end{aligned}
			\right.
		\]
		where $\kappa_1 \coloneqq (\eps_1 + i\sigma_1 / \omega) / \eps_0$. Since the boundary conditions are obviously satisfied on the accessible part $\Gamma_0$, we get $(\bff, \bfg) \in C(\eps_1, \sigma_1; \Gamma_0)$. Thus, $(\bff, \bfg) \in C(\eps_2, \sigma_2; \Gamma_0) = C(\eps_1, \sigma_1; \Gamma_0)$ by assumption. Therefore, there is a field $\bfE_2 \in H(\curl)$ such that
		\[
			\left\{
				\begin{aligned}
					\curl\curl\bfE_2 - k^2\kappa_2\bfE_2 &= 0, && \text{in } \Omega, \\
					\bfE_2 \times \bfn &= \bff, && \text{on } \Gamma_0, \\
					\curl\bfE_2 \times \bfn &= \bfg, && \text{on } \Gamma_0,
				\end{aligned}
			\right.
		\]
		where $\kappa_2 \coloneqq (\eps_2 + i\sigma_2 / \omega) / \eps_0$. Note that the boundary conditions are only satisfied on $\Gamma_0$. Now, we define $\bfE \coloneqq \bfE_1 - \bfE_2$. In the neighborhood $\mcV$ of $\Gamma$, $\bfE$ then satisfies
		\[
			\left\{
				\begin{aligned}
					\curl\curl\bfE - k^2\kappa\bfE &= 0, && \text{in } \mcV, \\
					\bfE \times \bfn &= 0, && \text{on } \Gamma_0, \\
					\curl\bfE \times \bfn &= 0, && \text{on } \Gamma_0,
				\end{aligned}
			\right.
		\]
		where $\kappa \coloneqq \kappa_1 = \kappa_2$ by assumption on $\mcV$. We then apply \cite[Lemma~5.4, ii]{Brownetal16} which yields $\bfE \equiv 0$ in $\bar{\mcV}$. Consequently, we get $\bfE_2 \times \bfn = \bfE_1 \times \bfn = \bff$ and $\curl\bfE_2 \times \bfn = \curl\bfE_1 \times \bfn = \bfg$ on the whole boundary $\Gamma$. Then, $(\bff, \bfg)$ belongs to the Cauchy data set $C(\eps_2, \sigma_2; \Gamma)$. Changing the roles of $(\eps_1, \sigma_1)$ and $(\eps_2, \sigma_2)$ proves that
		\[C(\eps_1, \sigma_1; \Gamma) = C(\eps_2, \sigma_2; \Gamma).\]
		Now, we infer from the assumptions on the coefficients that $\partial^\alpha \eps_1 = \partial^\alpha \eps_2$ and $\partial^\alpha \sigma_1 = \partial^\alpha \sigma_2$ on the boundary $\Gamma$ for any multi-index $\alpha \in \N^3$ such that $|\alpha| \leq 1$. These properties are the assumptions of the global uniqueness theorem of Caro and Zhou (see \cite[Theorem~1.1]{CaroZhou14}). This gives $\eps_1 = \eps_2$ and $\sigma_1 = \sigma_2$ in $\Omega$ and completes the proof.
	\end{proof}

	\section{Data transmission}
	\label{sec:transmission}

	In this section, we are interested in the data completion problem for Maxwell equation. The entry data for our inverse problem is the knowledge of the tangential trace $\bfg_D \coloneqq \bfE \times \bfn$ on the accessible part $\Gamma_0$, with $\bfE$ being the electric field resulting from the boundary source term $\bfg_N \coloneqq \curl\bfE \times \bfn$. The question of data completion then reads as follow: how can we compute the missing data, that is $\bfE \times \bfn$ on the whole boundary $\Gamma$?

	As the propagation of the electric field inside $\Omega$ is modeled by~\eqref{eq:maxwell}, it seems natural to consider the following Cauchy problem:
	\begin{equation}
		\label{eq:cauchy_Omega}
		\left\{
			\begin{aligned}
				\curl\curl\bfE - k^2\kappa\bfE &= 0, && \text{in } \Omega, \\
				\bfE \times \bfn &= \bfg_D, && \text{on } \Gamma_0, \\
				\curl\bfE \times \bfn &= \bfg_N, && \text{on } \Gamma_0.
			\end{aligned}
		\right.
	\end{equation}
	We could then define the missing boundary data as the traces of the solution of problem~\eqref{eq:cauchy_Omega}. However, solving this problem requires the knowledge of the refractive index $\kappa$ in the whole domain $\Omega$, which is unknown in the context of our inverse problem. To work around this issue, we use the main hypothesis stated on the refractive index: $\kappa$ is equal to the (known) background refractive index $\kappa_0$ in $\mcV$, the tubular neighborhood of $\Gamma$. This allows us to write that $\bfE$ satisfies the same Cauchy problem as defined above, but in $\mcV$, where the refractive index is known:
	\begin{equation}
		\label{eq:cauchy}
		\left\{
			\begin{aligned}
				\curl\curl\bfE - k^2\kappa_0\bfE &= 0, && \text{in } \mcV, \\
				\bfE \times \bfn &= \bfg_D, && \text{on } \Gamma_0, \\
				\curl\bfE \times \bfn &= \bfg_N, && \text{on } \Gamma_0.
			\end{aligned}
		\right.
	\end{equation}

	From \cite[Lemma~5.4, ii]{Brownetal16}, we know that, if problem~\eqref{eq:cauchy} admits a solution, then it is unique. However, Cauchy problems are known to be ill-posed \cite{Alessandrinietal09, BenBelgacem07}. Then, in order to solve problem~\eqref{eq:cauchy} numerically, it is necessary to regularize it. In \cite{DarbasHeleineLohrengel_QR}, we proposed to adapt the quasi-reversibility method to Maxwell's equations to solve problem~\eqref{eq:cauchy}. In particular, we proved the convergence of the relaxed and regularized formulation to the solution of this Cauchy problem in the case of noisy data. This method gives satisfying results. However, the question of how to choose the involved penalization parameter remains open: in the mentioned paper, we used the method of L-curves, and numerical experiments show that it is indeed promising, but there is no theoretical results to prove its convergence. Moreover, from a numerical point of view, using a penalization parameter that can be very small leads to linear systems that can be difficult to solve, especially in 3D domains.

	As an alternative, we propose here to use the iterated quasi-reversibility method, introduced in \cite{Darde16}. To this end, we introduce the auxiliary unknown $\bfF \coloneqq \curl\bfE$. In order to keep simple notations, we describe here only the 3D case. However, the method can also be written in 2D in a similar way, adapting the spaces by taking into account that there are two $\curl$ operators in this case. Then, problem~\eqref{eq:cauchy} can be written:
	\begin{equation}
		\label{eq:cauchy_F}
		\left\{
			\begin{aligned}
				\curl\bfF - k^2\kappa_0\bfE &= 0, && \text{in } \mcV, \\
				\curl\bfE - \bfF &= 0, && \text{in } \mcV, \\
				\bfE \times \bfn &= \bfg_D, && \text{on } \Gamma_0, \\
				\bfF \times \bfn &= \bfg_N, && \text{on } \Gamma_0.
			\end{aligned}
		\right.
	\end{equation}
	To take into account potential noise, both boundary data $\bfg_D$ and $\bfg_N$ are assumed to belong to $L^2(\Gamma_0)^3$. It is then natural to look for $\bfE$ and $\bfF$ in
	\[\tilde{H}(\curl) \coloneqq \set{\bfu \in L^2(\mcV)^3}{\curl\bfu \in L^2(\mcV)^3 \text{ and } \bfu \times \bfn \in L^2(\Gamma_0)^3},\]
	which is a Hilbert space, endowed with the scalar product
	\[\forall (\bfu,\bfv) \in \tilde{H}(\curl)^2, \quad \dotprod{\bfu}{\bfv}[\tilde{H}(\curl)] \coloneqq \int_\mcV (\bfu \cdot \bar{\bfv} + \curl\bfu \cdot \curl\bar{\bfv}) + \int_{\Gamma_0} (\bfu \times \bfn) \cdot (\bar{\bfv} \times \bfn)\]
	and the induced norm $\norm{}[\tilde{H}(\curl)]^2 \coloneqq \dotprod{}{}[\tilde{H}(\curl)]$.

	Let $\mcX \coloneqq \tilde{H}(\curl)^2$ and $\mcY \coloneqq L^2(\mcV)^3 \times L^2(\mcV)^3 \times L^2(\Gamma_0)^3 \times L^2(\Gamma_0)^3$. These are Hilbert spaces endowed with their respective graph norms. Let us introduce the operator
	\[
		\begin{aligned}
			A \colon \hspace{2em} \mcX & \to \mcY \\
			(\bfE,\bfF) & \mapsto (\curl\bfF - k^2\kappa_0\bfE, \curl\bfE - \bfF, \restriction{\bfE}{\Gamma_0} \times \bfn, \restriction{\bfF}{\Gamma_0} \times \bfn)
		\end{aligned}
	\]
	so that problem~\eqref{eq:cauchy_F} can be rewritten:
	\[\text{Find } \bfx = (\bfE,\bfF) \in \mcX \text{ such that } A\bfx = \bfy \coloneqq (0, 0, \bfg_D, \bfg_N) \in \mcY.\]

	\begin{proposition}
		\label{prop:properties_A}
		The operator $A$ is linear, continuous, one-to-one, not onto and has a dense range.
	\end{proposition}

	\begin{proof}
		It is clear that $A$ is linear and continuous. As the Cauchy problem~\eqref{eq:cauchy} admits at most one solution but may have no solution, $A$ is one-to-one but not onto. Let us prove that its range is dense in $\mcY$. To this end, we will prove that $\range(A)^\perp = \{0\}$. Let $\bfy = (\bfu,\bfv,\bff,\bfg) \in \range(A)^\perp$. In other words, we have :
		\[\forall (\bfE,\bfF) \in \mcX, \quad \dotprod{A(\bfE,\bfF)}{(\bfu,\bfv,\bff,\bfg)}[\mcY] = 0,\]
		which reads:
		\[\forall (\bfE,\bfF) \in \mcX, \int_\mcV ((\curl\bfF - k^2\kappa_0\bfE) \cdot \bar{\bfu} + (\curl\bfE - \bfF) \cdot \bar{\bfv}) + \int_{\Gamma_0} ((\bfE \times \bfn) \cdot \bar{\bff} + (\bfF \times \bfn) \cdot \bar{\bfg}) = 0.\]
		We first consider $\bfF \in \mcC^\infty_c(\mcV)^3$ and $\bfE = 0$, which yields
		\[\int_\mcV (\curl\bfF \cdot \bar{\bfu} - \bfF \cdot \bar{\bfv}) = 0.\]
		An integration by parts then gives $\curl\bfu = \bfv$ in $\mcV$. As $\bfv \in L^2(\mcV)^3$, this yields $\bfu \in H(\curl; \mcV)$. By considering $\bfE \in \mcC^\infty_c(\mcV)^3$ and $\bfF = 0$, a similar argument gives $\curl\bfv = k^2\overline{\kappa_0}\bfu$ and then $\bfv \in H(\curl; \mcV)$. Now, we take $\bfE = 0$ and $\bfF \in \tilde{H}(\curl)$. An integration by parts yields
		\[\duality{\bfF \times \bfn}{\bfn \times (\bfu \times \bfn)}[\partial\mcV] = \int_{\Gamma_0} (\bfF \times \bfn) \cdot \bar{\bfg}.\]
		Then, we obtained:
		\[\bfg = \begin{cases}
			\bfn \times (\bfu \times \bfn), & \text{on } \Gamma_0, \\
			0, & \text{on } \partial\mcV \setminus \Gamma_0.
		\end{cases}\]
		Similarly, by taking $\bfE \in \tilde{H}(\curl)$ and $\bfF = 0$, we get
		\[\bff = \begin{cases}
			\bfn \times (\bfv \times \bfn), & \text{on } \Gamma_0, \\
			0, & \text{on } \partial\mcV \setminus \Gamma_0.
		\end{cases}\]
		Grouping all these results together, we obtain that $\bfu \in H(\curl; \mcV)$ satisfies
		\[
			\left\{
				\begin{aligned}
					\curl\curl\bfu - k^2\overline{\kappa_0}\bfu &= 0, && \text{in } \mcV, \\
					\bfu \times \bfn &= 0, && \text{on } \partial\mcV \setminus \Gamma_0, \\
					\curl\bfu \times \bfn &= 0, && \text{on } \partial\mcV \setminus \Gamma_0. \\
				\end{aligned}
			\right.
		\]
		By uniqueness of the solution to the Cauchy problem, we then have $\bfu \equiv 0$, and then $\bfv \equiv 0$, $\bff \equiv 0$ and $\bfg \equiv 0$, so $\bfy = 0$, which ends the proof.
	\end{proof}

	Let us now introduce the sesquilinear form
	\[
		\begin{aligned}
			b \colon \hspace{5em} \mcX \times \mcX & \to \C \\
			((\bfu_1,\bfv_1), (\bfu_2,\bfv_2)) & \mapsto \int_\Omega (\bfu_1 \cdot \bar{\bfv_1} + \bfu_2 \cdot \bar{\bfv_2}).
		\end{aligned}
	\]
	It is readily seen that the norm $\sqrt{\norm{A\cdot}[\mcY]^2 + b(\cdot,\cdot)}$ is equivalent to $\norm{}[\mcX]$. Let $\bfy = (0, 0, \bfg_D, \bfg_N) \in \mcY$ be such that problem~\eqref{eq:cauchy_F} admits a (unique) solution $(\bfE,\bfF)$. For $\delta > 0$, we define the sequence $(\bfx_\delta^M)_{M \in \N \cup \{-1\}}$ by $\bfx_\delta^{-1} \coloneqq 0$ and then, for all $M \in \N$, $x_\delta^M$ is the unique element in $\mcX$ satisfying:
	\[\forall \bfx \in \mcX, \quad \dotprod{A\bfx_\delta^M}{A\bfx}[\mcY] + \delta b(\bfx_\delta^M, \bfx) = \dotprod{\bfy}{A\bfx}[\mcY] + \delta b({\bfx_\delta^{M-1}, \bfx}).\]
	Results from \cite{Darde16} then guarantees, using \autoref{prop:properties_A}, that
	\[\lim_{M \to +\infty} \bfx_\delta^M = (\bfE,\bfF).\]

	As it has been observed in \cite{DarbasHeleineLohrengel_QR}, the error committed when the solution of the Cauchy problem is approximated using the quasi-reversibility method is lower in the interior of the domain $\mcV$ than on the nonaccessible boundary $\Gamma_1$. Then, in view of solving our inverse problem, it seems interesting to not complete the data on the whole boundary $\Gamma$, which requires to use the reconstructed field on $\Gamma_1$. Instead, we introduce an artificial boundary inside $\mcV$, $\Gamma_\text{int} \coloneqq \partial U$, where $U$ is a open set such that $U \cap \mcV \neq \emptyset$, as illustrated in \autoref{fig:configuration}. The quasi-reversibility method allows us to transmit the partial data from $\Gamma_0$ to $\Gamma_\text{int}$: then, we can consider that we have total data on the boundary of the subdomain $U$.

	\begin{figure}[hbt]
		\centering

		\begin{tikzpicture}[scale = 0.8]
			\useasboundingbox (-5.5,-3) rectangle (5.5,5.5);

			\newcommand{\domainboundarycoordinates}{
				(5,0) .. controls +(0,2) and +(1,-1) .. (3,3)
				.. controls +(-1,1) and +(1,0) .. (0,5)
				.. controls +(-3,0) and +(0,1) .. (-5,0)
				.. controls +(0,-2) and +(-2,0) .. (0,-3)
				.. controls +(3,0) and +(0,-2) .. cycle
			}

			\draw[pattern = vertical lines, even odd rule]
				\domainboundarycoordinates

				(3,0) .. controls +(0,2) and +(1,0) .. (0,3.5)
				.. controls +(-2,0) and +(0,1) .. (-3,0)
				.. controls +(0,-1) and +(-2,0) .. (0,-2)
				.. controls +(1,0) and +(0,-2) .. cycle
			;

			\begin{scope}
				\clip
					(3,3) circle (1)
					(-2.403,4.109) circle (1)
					(-5,0) circle (1)
					(-1.962,-2.73) circle (1)
					(3.625,-2.25) circle (1)
				;
				\draw[line width = 1.5mm]
					\domainboundarycoordinates
				;
			\end{scope}

			\draw[dashed, fill = gray, fill opacity = 0.2]
				(4,0) .. controls +(0,1.5) and +(2,0) .. (0,4)
				.. controls +(-2,0) and +(0,1.5) .. (-3.8,0)
				.. controls +(0,-2) and +(-2,0) .. (0,-2.5)
				.. controls +(2,0) and +(0,-2) .. cycle
			;

			\begin{scope}[dotted]
				\draw[rotate around = {-45:(0.5,2)}]
					(0.5,1.5) .. controls +(0.1,0) and +(0,-0.1) .. (0.7,1.7)
					.. controls +(0,0.1) and +(0,-0.1) .. (0.6,2)
					.. controls +(0,0.1) and +(0,-0.1) .. (0.7,2.3)
					.. controls +(0,0.1) and +(0.1,0) .. (0.5,2.5)
					.. controls +(-0.1,0) and +(0,0.1) .. (0.3,2.3)
					.. controls +(0,-0.1) and +(0,0.1) .. (0.4,2)
					.. controls +(0,-0.1) and +(0,0.1) .. (0.3,1.7)
					.. controls +(0,-0.1) and +(-0.1,0) .. cycle
				;

				\draw
					(-1.5,-0.5) circle (0.5)
				;

				\draw
					(1.5,-1.5) .. controls +(1,0) and +(1,0) .. (2,-0.5)
					.. controls +(-0.2,0) and +(0.1,0.1) .. (1.7,-1)
					.. controls +(-0.2,-0.2) and +(0,0.5) .. (1,-1.3)
					.. controls +(0,-0.2) and +(-0.3,0) .. cycle
				;
			\end{scope}

			\node at (-2.65,4.35) [rotate = 45] {$\Gamma_0$};
			\node at (4.3,0) [rotate = -90] {$\Gamma_\text{int}$};
			\node at (-3.8,0) [node font = \LARGE] {$\mcV$};
			\node at (2.3, 1.7) [node font = \LARGE] {$U$};
		\end{tikzpicture}

		\caption{Example of possible configuration for the domain $\Omega$. The accessible part of the boundary $\Gamma_0$ is shown as thick parts on $\partial\Omega$. The support of the perturbation is here composed of three parts, delimited by dotted lines. This support does not touch $\mcV$, the tubular neighborhood where $\kappa$ is assumed to be known, represented here by the part filled with vertical lines. Finally, $\Gamma_\text{int}$ is an artificial boundary included in $\mcV$, delimiting the subdomain $U$ represented by the gray part.}
		\label{fig:configuration}
	\end{figure}

	\begin{remark}
		When using the iterated quasi-reversibility method to approximate the solution of the rewritten Cauchy problem~\eqref{eq:cauchy_F}, one can observe that the error can be too important to be neglected in high-frequency regimes. One possible workaround, if $\kappa_0$ is constant over $\mcV$, consists in considering $\bfF = \frac{1}{k\sqrt{\kappa_0}}\curl\bfE$ instead of simply $\bfF = \curl\bfE$. This yields the system
		\[
			\left\{
				\begin{aligned}
					\curl\bfF - k\sqrt{\kappa_0}\bfE &= 0, && \text{in } \mcV, \\
					\curl\bfE - k\sqrt{\kappa_0}\bfF &= 0, && \text{in } \mcV, \\
					\bfE \times \bfn &= \bfg_D, && \text{on } \Gamma_0, \\
					\bfF \times \bfn &= \frac{1}{k\sqrt{\kappa_0}}\bfg_N, && \text{on } \Gamma_0,
				\end{aligned}
			\right.
		\]
		offering a better precision.
	\end{remark}

	\section{Iterated sensitivity equation}
	\label{sec:sensitivity}

	From now on, for any admissible refractive index $\kappa$, we denote by $\bfE[\kappa]$ the solution of the direct problem
	\begin{equation}
		\label{eq:direct}
		\left\{
			\begin{aligned}
				\curl\curl\bfE[\kappa] - k^2\kappa\bfE[\kappa] &= 0, && \text{in } \Omega, \\
				\curl\bfE[\kappa] \times \bfn &= \bfg_N, && \text{on } \Gamma,
			\end{aligned}
		\right.
	\end{equation}
	where $\bfg_N$ is the Neumann trace of a given incident wave. In the rest of this paper, we adopt the following notation: for any couple $(a,b) \in (\R^+)^2$,
	\[a \lesssim b \iff (\exists C > 0 / a \leq Cb).\]

	\begin{proposition}
		\label{prop:solution_maxwell}
		Let $\bfF \in L^2(\Omega)^p$ and $\bfg \in Y(\Gamma)'$. The variational formulation of the problem
		\[
			\left\{
				\begin{aligned}
					\curl\curl\bfE - k^2\kappa\bfE &= \bfF, && \text{in } \Omega, \\
					\curl\bfE \times \bfn &= \bfg, && \text{on } \Gamma,
				\end{aligned}
			\right.
		\]
		reads
		\[
			\left\{
				\begin{aligned}
					&\text{Find } \bfE \in H(\curl) \text{ such that} \\
					&\dotprod{\curl\bfE}{\curl\bfphi} - k^2\dotprod{\kappa\bfE}{\bfphi} = \dotprod{\bfF}{\bfphi} + \duality{\bfg}{\bfphi_T}[\Gamma], \quad \forall \bfphi \in H(\curl),
				\end{aligned}
			\right.
		\]
		where $\dotprod{}{}$ denotes the dot product in $L^2(\Omega)^p$, $\duality{}{}[\Gamma]$ denotes the duality bracket from $Y(\Gamma)'$ to $Y(\Gamma)$ and $\bfphi_T$ is the continuous extension of the map $\bfphi \mapsto \bfn \times (\bfphi \times \bfn)$. This variational problem admits a unique solution $\bfE \in H(\curl)$, depending continuously on $\bfF$ and $\bfg$:
		\[\norm{\bfE}[H(\curl)] \lesssim \norm{\bfF}[0] + \norm{\bfg}[Y(\Gamma)'],\]
		where $\norm{}[0]$ denotes the $L^2(\Omega)^p$-norm.
	\end{proposition}

	\begin{remark}
		Under the assumptions described in \autoref{def:admissible}, one can prove that the sesquilinear form involved in the variational formulation is coercive (see for example \cite{HazardLenoir96} or \cite[Remark~1.6]{Heleine19}), hence the result. It is possible to consider weaker assumptions on the refractive index, allowing for example the conductivity to be zero on some parts of $\Omega$ (as long as it is not zero everywhere), as shown in \cite{Monk03}. To remain as simple as possible in the statement of \autoref{thm:sensitivity} which follows, we chose to not consider such setting.
	\end{remark}

	We recall here that we are looking for an unknown refractive index which is assumed to be a perturbation of a known background index $\kappa_0$. We will see later that the idea of the inverse procedure we propose in this paper is to separate the informations from the support and amplitude of the perturbations. In this section, we then temporarily fix the support $D$ of this perturbation and then only its amplitude, $a \in \R$, is allowed to vary. We will then denote $\kappa_a$ the refractive index:
	\[\kappa_a \coloneqq \kappa_0(1 + a\chi_D),\]
	where $\chi_D$ is the characteristic function of $D$, for any $a \in \R$ such that $\kappa_a$ is still an admissible refractive index. The goal of this section is to provide a tool to study small variations around this amplitude. For any given real number $h$, we then introduce the perturbed refractive index
	\[\kappa_{a+h} \coloneqq \kappa_0(1 + (a + h)\chi_D) = \kappa_a + h\kappa_0\chi_D.\]
	To study these small variations, we introduce the application mapping an amplitude $a$ of perturbation to the corresponding electric field $\bfE[\kappa_a]$, solution of \eqref{eq:direct}:
	\[
		\begin{aligned}
			K \colon \R & \to H(\curl) \\
			a & \mapsto \bfE[\kappa_a].
		\end{aligned}
	\]

	\begin{theorem}
		\label{thm:sensitivity}
		Let $I \subset \R$ be an open interval such that $\kappa_a$ is admissible for all $a \in I$. Then $K \in \mcC^\infty(I; H(\curl))$. Moreover, if we denote, for any $a \in I$ and $n \in \N$, $\bfE_a^{(n)} \coloneqq K^{(n)}(a)$ the $n$-th derivative of $K$ at point $a$, then $\bfE_a^{(0)} = \bfE[\kappa_a]$ and, for all $n \in \N^*$, $\bfE_a^{(n)}$ is the unique solution of the iterated sensitivity equation:
		\begin{equation}
			\label{eq:sensitivity}
			\left\{
				\begin{aligned}
					&\text{Find } \bfE_a^{(n)} \in H(\curl) \text{ such that} \\
					&\dotprod{\curl \bfE_a^{(n)}}{\curl\bfphi} - k^2\dotprod{\kappa_a \bfE_a^{(n)}}{\bfphi} = nk^2\dotprod{\kappa_0\chi_D \bfE_a^{(n-1)}}{\bfphi}, \quad \forall \bfphi \in H(\curl).
				\end{aligned}
			\right.
			\tag{$\mcS$}
		\end{equation}
	\end{theorem}

	\begin{proof}
		Let $a \in I$. As $I$ is an open interval, there exists $h_0 > 0$ such that $a + h \in I$ for all $h \in \interval[open]{-h_0}{h_0}$. In the following, $h$ will always be assumed to belong in this interval. We will prove by induction on $n \in \N^*$ that:
		\begin{equation}
			\label{hyp:induction_sensitivity}
			\norm{\bfE_{a+h}^{(n-1)} - \bfE_a^{(n-1)}}[H(\curl)] \lesssim \abs{h}, \quad \norm{\bfE_{a+h}^{(n-1)}}[H(\curl)] \lesssim \norm{\bfg_N}[\Gamma], \quad \norm{\tilde{\bfE}_a^{(n)} - \bfE_a^{(n)}}[H(\curl)] \lesssim \abs{h},
			\tag{$\mcH_n$}
		\end{equation}
		where:
		\[\forall n \in \N^*, \quad \tilde{\bfE}_a^{(n)} \coloneqq \frac{\bfE_{a+h}^{(n-1)} - \bfE_a^{(n-1)}}{h}.\]

		Let begin with the case $n = 1$. As $\bfE_a^{(0)} = \bfE[\kappa_a]$ is the weak solution of problem~\eqref{eq:direct}, it satisfies (according to \autoref{prop:solution_maxwell}):
		\begin{equation}
			\label{eq:Ea0}
			\dotprod{\curl\bfE_a^{(0)}}{\curl\bfphi} - k^2\dotprod{\kappa_a\bfE_a^{(0)}}{\bfphi} = \duality{\bfg_N}{\bfphi_T}[\Gamma], \quad \forall \bfphi \in H(\curl).
		\end{equation}
		We can as well write the variational problem satisfied by $\bfE_{a+h}^{(0)}$:
		\[\dotprod{\curl\bfE_{a+h}^{(0)}}{\curl\bfphi} - k^2\dotprod{\kappa_{a+h}\bfE_{a+h}^{(0)}}{\bfphi} = \duality{\bfg_N}{\bfphi_T}[\Gamma], \quad \forall \bfphi \in H(\curl),\]
		from which we can directly infer, as a consequence of \autoref{prop:solution_maxwell}, that $\norm{\bfE_{a+h}^{(0)}}[H(\curl)] \lesssim \norm{\bfg_N}[\Gamma]$. We rewrite this problem using the identity $\kappa_{a+h} = \kappa_a + h\kappa_0\chi_D$:
		\begin{equation}
			\label{eq:Eah0}
			\dotprod{\curl\bfE_{a+h}^{(0)}}{\curl\bfphi} - k^2\dotprod{\kappa_a\bfE_{a+h}^{(0)}}{\bfphi} = \duality{\bfg_N}{\bfphi_T}[\Gamma] + k^2h\dotprod{\kappa_0\chi_D\bfE_{a+h}^{(0)}}{\bfphi}, \quad \forall \bfphi \in H(\curl).
		\end{equation}
		Computing the difference between \eqref{eq:Eah0} and \eqref{eq:Ea0} yields:
		\begin{equation}
			\label{eq:diff_Eah0-Ea0}
			\dotprod{\curl(\bfE_{a+h}^{(0)} - \bfE_a^{(0)})}{\curl\bfphi} - k^2\dotprod{\kappa_a(\bfE_{a+h}^{(0)} - \bfE_a^{(0)})}{\bfphi} = k^2h\dotprod{\kappa_0\chi_D\bfE_{a+h}^{(0)}}{\bfphi}, \quad \forall \bfphi \in H(\curl).
		\end{equation}
		From \eqref{eq:diff_Eah0-Ea0} and \autoref{prop:solution_maxwell}, we obtain:
		\[\norm{\bfE_{a+h}^{(0)} - \bfE_a^{(0)}}[H(\curl)] \lesssim k^2\abs{h}\norm{\kappa_0\chi_D\bfE_{a+h}^{(0)}}[0] \lesssim \abs{h}\norm{\bfE_{a+h}^{(0)}}[H(\curl)] \lesssim \abs{h},\]
		which, in particular, shows that $K$ is continuous at point $a$. Now, from the definition of $\tilde{\bfE}_a^{(1)}$, using \eqref{eq:diff_Eah0-Ea0} and computing the difference with \eqref{eq:sensitivity} at rank $n = 1$, we get:
		\[\dotprod{\curl(\tilde{\bfE}_a^{(1)} - \bfE_a^{(1)})}{\curl\bfphi} - k^2\dotprod{\kappa_a(\tilde{\bfE}_a^{(1)} - \bfE_a^{(1)})}{\bfphi} = k^2\dotprod{\kappa_0\chi_D(\bfE_{a+h}^{(0)} - \bfE_a^{(0)})}{\bfphi}, \quad \forall \bfphi \in H(\curl),\]
		from which, once again using \autoref{prop:solution_maxwell}, we get:
		\[\norm{\tilde{\bfE}_a^{(1)} - \bfE_a^{(1)}}[H(\curl)] \lesssim \norm{\bfE_{a+h}^{(0)} - \bfE_a^{(0)}}[H(\curl)] \lesssim \abs{h}.\]
		This shows that
		\[\lim_{h \to 0} \frac{K(a+h) - K(a)}{h} = \bfE_a^{(1)}\]
		in $H(\curl)$.

		Now, let $n \in \N^*$ be such that \eqref{hyp:induction_sensitivity} is true. We will show that $K^{(n)}$ is still continuous and differentiable at point $a$. To this end, let us write the variational problems satisfied by $\bfE_a^{(n)}$ and $\bfE_{a+h}^{(n)}$:
		\begin{align*}
			\dotprod{\curl\bfE_a^{(n)}}{\curl\bfphi} - k^2\dotprod{\kappa_a\bfE_a^{(n)}}{\bfphi} &= nk^2\dotprod{\kappa_0\chi_D\bfE_a^{(n-1)}}{\bfphi}, \quad \forall \bfphi \in H(\curl), \\
			\dotprod{\curl\bfE_{a+h}^{(n)}}{\curl\bfphi} - k^2\dotprod{\kappa_{a+h}\bfE_{a+h}^{(n)}}{\bfphi} &= nk^2\dotprod{\kappa_0\chi_D\bfE_{a+h}^{(n-1)}}{\bfphi}, \quad \forall \bfphi \in H(\curl).
		\end{align*}
		In particular, this gives:
		\[\norm{\bfE_{a+h}^{(n)}}[H(\curl)] \lesssim \norm{\bfE_{a+h}^{(n-1)}}[H(\curl)] \lesssim \norm{\bfg_N}[\Gamma],\]
		using \autoref{prop:solution_maxwell} and \eqref{hyp:induction_sensitivity}. As in the case $n = 1$, we use the identity $\kappa_{a+h} = \kappa_a + h\kappa_0\chi_D$ to get:
		\begin{multline*}
			\dotprod{\curl(\bfE_{a+h}^{(n)} - \bfE_a^{(n)})}{\curl\bfphi} - k^2\dotprod{\kappa_a(\bfE_{a+h}^{(n)} - \bfE_a^{(n)})}{\bfphi} \\
			= nk^2\dotprod{\kappa_0\chi_D(\bfE_{a+h}^{(n-1)} - \bfE_a^{(n-1)})}{\bfphi} + k^2h\dotprod{\kappa_0\chi_D\bfE_{a+h}^{(n)}}{\bfphi}, \quad \forall \bfphi \in H(\curl).
		\end{multline*}
		Then, using \eqref{hyp:induction_sensitivity}:
		\[\norm{\bfE_{a+h}^{(n)} - \bfE_a^{(n)}}[H(\curl)] \lesssim \norm{\bfE_{a+h}^{(n-1)} - \bfE_a^{(n-1)}}[H(\curl)] + \abs{h}\norm{\bfE_{a+h}^{(n)}}[H(\curl)] \lesssim \abs{h},\]
		and then $K^{(n)}$ is continuous at point $a$. Using the definition of $\tilde{\bfE}_a^{(n+1)}$ and \eqref{eq:sensitivity} written at rank $n+1$, we get:
		\begin{multline*}
			\dotprod{\curl(\tilde{\bfE}_a^{(n+1)} - \bfE_a^{(n+1)})}{\curl\bfphi} - k^2\dotprod{\kappa_a(\tilde{\bfE}_a^{(n+1)} - \bfE_a^{(n+1)})}{\bfphi} \\
			= nk^2\dotprod{\kappa_0\chi_D(\tilde{\bfE}_a^{(n)} - \bfE_a^{(n)})}{\bfphi} + k^2\dotprod{\kappa_0\chi_D(\bfE_{a+h}^{(n)} - \bfE_a^{(n)})}{\bfphi}, \quad \forall \bfphi \in H(\curl).
		\end{multline*}
		Then, as usual, we use \autoref{prop:solution_maxwell} and \eqref{hyp:induction_sensitivity} to obtain
		\[\norm{\tilde{\bfE}_a^{(n+1)} - \bfE_a^{(n+1)}}[H(\curl)] \lesssim \abs{h},\]
		which concludes the proof.
	\end{proof}

	As we are interested in the tangential trace of the electric field, it is natural to also define the application mapping the amplitude of the perturbation to the trace of the corresponding field:
	\[
		\begin{aligned}
			K_t \colon \R & \to Y(\Gamma) \\
			a & \mapsto \gamma_t(\bfE[\kappa_a]).
		\end{aligned}
	\]
	The following result is a direct consequence of \autoref{thm:sensitivity} and of the fact that $\gamma_t$ is linear and continuous from $H(\curl)$ to $Y(\Gamma)$.

	\begin{corollary}
		\label{cor:Kt}
		Let $I \subset \R$ be an open interval such that $\kappa_a$ is admissible for all $a \in I$. Then $K_t \in \mcC^\infty(I; Y(\Gamma))$. Moreover, for any $a \in I$ and $n \in \N$, we have $K_t^{(n)}(a) = \gamma_t \circ K^{(n)}(a)$.
	\end{corollary}

	\section{A linearized cost functional}
	\label{sec:cost}

	We recall that our goal in this paper is to reconstruct a refractive index $\kappa_\text{ex}$. When $\Omega$ is illuminated by a wave of trace $\bfg_N$, the electric field $\bfE[\kappa_\text{ex}]$ propagating inside $\Omega$ is solution of problem~\eqref{eq:direct}. In the context of our inverse problem, $\bfE[\kappa_\text{ex}]$ is unknown in $\Omega$. Instead, we have access to its tangential trace $\bfE[\kappa_\text{ex}] \times \bfn$ on $\Gamma_0$, the accessible part of the boundary. To solve this inverse problem, a natural approach is to find $\kappa$ minimizing the cost functional
	\[\mcJ \colon \kappa \mapsto \frac{1}{2}\norm{\bfE[\kappa_\text{ex}] \times \bfn - \bfE[\kappa] \times \bfn}[\Gamma_0]^2,\]
	measuring the distance between the measurements and the predictions of the direct problem. Unfortunately, trying to minimize $\mcJ$ without any regularization leads to poor results (see for example \cite[Section~6.3]{Heleine19}). To regularize this functional, we will first recall that we are assuming that $\kappa_\text{ex}$ is a perturbation of a known background refractive index $\kappa_0$. The only hypothesis about the support of the perturbation is that it is contained outside a given tubular neighborhood of the domain's boundary. Then, the support of this perturbation, denoted $D_\text{ex}$, is unknown, as well as its amplitude $a_\text{ex} \in \R$. With these notations, we have:
	\[\kappa_\text{ex} = \kappa_0(1 + a_\text{ex}\chi_{D_\text{ex}}).\]
	Our goal is then to retrieve $a_\text{ex}$ and $D_\text{ex}$, which is then equivalent to the reconstruction of $\kappa_\text{ex}$.

	The applications $K$ and $K_t$ defined in \autoref{sec:sensitivity} can actually be defined for any support of perturbation $D$. To highlight this dependency, we rename them to $K_D$ and $K_{D,t}$: each choice of support defines a new couple of applications. This notation allows us to define a first regularization of $\mcJ$ where we are explicitly looking for a perturbation of $\kappa_0$:
	\[\mcJ_r \colon (D, a) \mapsto \frac{1}{2}\norm{\bfE[\kappa_\text{ex}] \times \bfn - K_{D,t}(a)}[\Gamma_0]^2.\]
	This naturally regularizes the cost functional by reducing the size of the set of allowed refractive indices. However, evaluating $\mcJ_r$ still requires solving problem~\eqref{eq:direct}, each time with a different refractive index. This means that, each time we evaluate $\mcJ_r$ numerically, we have a new matrix to invert, which can lead to high computational cost and time. We propose here a way to work around this issue by assuming that the amplitude of the perturbation is such that $\abs{a} < 1$. Then, we can use \autoref{cor:Kt}. For any fixed support $D$, as $K_{D,t}$ is $\mcC^\infty$, it admits a Taylor expansion at any order $N \in \N$:
	\[K_{D,t}(a) = \sum_{n=0}^N \frac{a^n}{n!}K_{D,t}^{(n)}(0) + o(a^N).\]
	It is then possible to choose $N \in \N^*$ large enough so that $o(a^N)$ can be neglected and:
	\[K_{D,t}(a) \approx K_{D,t}(0) + \sum_{n=1}^N \frac{a^n}{n!}K_{D,t}^{(n)}(0).\]
	The idea is then simply to replace the term $K_{D,t}(a)$ in the expression of $\mcJ_r$ by this truncated expansion:
	\[J \colon (D,a) \mapsto \frac{1}{2}\norm*{\bfE[\kappa_\text{ex}] \times \bfn - \bfE[\kappa_0] \times \bfn - \sum_{n=1}^N \frac{a^n}{n!}K_{D,t}^{(n)}(0)}[\Gamma_0]^2.\]
	As the background refractive index $\kappa_0$ is assumed to be known, $\bfE[\kappa_0]$ can always be computed and we can equivalently say that our entry data in our inverse problem is $\delta\bfE \times \bfn \coloneqq \bfE[\kappa_\text{ex}] \times \bfn - \bfE[\kappa_0] \times \bfn$. One can notice that the system required to be solved to compute the derivatives $K_{D,t}^{(n)}(0)$, for $n \in \N^*$ is the same as the one for $\bfE[\kappa_0]$, that is problem~\eqref{eq:direct} with the background refractive index $\kappa_0$, for any support $D$ or amplitude $a$. It is then faster to evaluate numerically $J$ than $\mcJ_r$, as long as a factorization (e.g. an LU decomposition) of the matrix is stored.

	\section{Numerical results}
	\label{sec:numerics}

	\subsection{A complete inversion procedure}

	In this section, we test the reconstruction of a refractive index $\kappa_\text{ex}$ in different configurations. We assume that $\kappa_\text{ex}$ is a perturbation of a known background refractive index $\kappa_0$. In the first tests, the support of the perturbation $D_\text{ex}$ is assumed to be a ball contained in $\Omega \setminus \bar{\mcV}$ where $\mcV$ will be a given tubular neighborhood of $\Gamma$. The amplitude of the perturbation is a real constant $a_\text{ex}$. With these notations, we are then looking for $D_\text{ex}$ and $a_\text{ex}$ such that:
	\[\kappa_\text{ex} = \kappa_0(1 + a_\text{ex}\chi_{D_\text{ex}}).\]
	When the support of the perturbation is assumed to be a ball, it is determined by two parameters: its center, denoted $\bfx_\text{ex}$ and its radius, denoted $r_\text{ex}$.

	The data is generated by solving problem~\eqref{eq:direct} with refractive index $\kappa_\text{ex}$, and with $\bfg$ being the Neumann trace of a plane wave $\bfE_\bfeta$ of direction $\bfeta \in \R^p$ (unitary):
	\[\bfE_\bfeta \colon \bfx \mapsto \bfeta^\perp e^{ik\sqrt{\kappa_0}\bfeta \cdot \bfx},\]
	where $\bfeta^\perp \in \R^p$ is a unitary vector orthogonal to $\bfeta$. Once the electric field $\bfE[\kappa_\text{ex}]$ is generated, its tangential trace $\bfE[\kappa_\text{ex}] \times \bfn$ is computed on a given part $\Gamma_0$ of the boundary. The entry data for the inverse procedure is the collection of the tangential traces obtained from plane waves of different directions $\bfeta$.

	As a first step, we transmit these partial data to a boundary $\Gamma_\text{int}$ defined inside $\mcV$, using the iterated quasi-reversibility method described in \autoref{sec:transmission}. More precisely, we transmit the difference field to get an approximation of $\restriction{(\bfE[\kappa_\text{ex}] - \bfE[\kappa_0])}{\Gamma_\text{int}} \times \bfn$, used in the rest of the algorithm to reconstruct the perturbation.

	In order to reduce the number of unknowns, we use a result from \cite{DarbasHeleineLohrengel_sens}: for $\bfz \in \Gamma_\text{int}$, the quantity $\abs{((\bfE[\kappa_\text{ex}] - \bfE[\kappa_0]) \times \bfn)(\bfz)}$ decreases as the distance between $\bfz$ and $\bfx_\text{ex}$ increases. In other words, this modulus is controled by a surface peak that can be localized to obtain an approximation of $\hat{\bfx}_\text{ex}$, the projection of $\bfx_\text{ex}$ on $\Gamma_\text{int}$. Denoting $\hat{\bfn}$ the normal at this point, we then have:
	\[\bfx_\text{ex} = \hat{\bfx}_\text{ex} - d_\text{ex}\hat{\bfn},\]
	where $d_\text{ex} > 0$ is the depth of the perturbation.

	Having an approximation of $\hat{\bfx}_\text{ex}$ allows to reduce the number of parameters to retrieve to three scalars: the depth of the perturbation, its radius, and its amplitude. For $d > 0$, we note $\bfx(d) = \hat{\bfx} - d\hat{\bfn}$. Let's denote $D(d,r)$ the ball centered at $\bfx(d)$ and of radius $r$. As a final step of the algorithm, we will then minimize the functional
	\[j \colon (d,r,a) \mapsto J(D(d,r), a).\]
	In all tests, we minimize in two steps. Let us introduce the functional
	\[\tilde{j} \colon (d,r) \mapsto \min_a j(d, r, a).\]
	In other words, for a given support of perturbation, we are looking for the optimal amplitude, that is the amplitude allowing us to best approach the boundary data. As the parameters $d$ and $r$ are fixed, the functional $j(d, r, \cdot)$ is a simple function from $\R$ to $\R$ and can then be minimized using classical methods. Here, we chose to use a golden section method~\cite{Kiefer53}. To minimize the functional $\tilde{j}$ which is defined from $\R^2$ to $\R$, we chose to apply Powell's method~\cite{Powell64}.

	All PDEs are solved with FreeFem++~\cite{Hecht12}, using edge finite elements of order 1 (see \cite{Nedelec86}). The meshes are generated with Gmsh~\cite{GeuzaineRemacke09}. All examples can be reproduced using the code available in the dedicated Git repository~\cite{Heleine23}.

	\subsection{Unit disk}
	\label{sec:numerics_unit_disk}

	We begin with a simple case, in two dimensions ($p = 2$). Here, $\Omega$ is the unit disk. The tubular neighborhood $\mcV$ where it is assumed that there is no perturbations is the annulus of outer radius 1 and of inner radius $0.7$. The accessible part of the boundary, $\Gamma_0$, is the union of 32 patches equally distributed on $\Gamma$. We add an artificial boundary $\Gamma_\text{int}$ inside $\mcV$ where the partial data will be transmitted: it is chosen as the circle centered at the origin and of radius $0.8$. Then, $U$ is the disk centered at the origin and of radius $0.8$. The perturbation that we will want to retrieve is located in the disk centered at point $(-0.4,0)$ and of radius $0.2$. This configuration is illustrated in \autoref{fig:configuration_disk}.

	\begin{figure}[hbt]
		\centering

		\begin{tikzpicture}[scale = 3]
			\newcommand{\domainboundarycoordinates}{
				(0,0) circle (1)
			}

			\draw[fill = gray, fill opacity = 0.1, even odd rule]
				\domainboundarycoordinates

				(0,0) circle (0.7)
			;

			\begin{scope}
				\clip
					(1,0) circle (0.075)
					(0.9807852804032304,0.19509032201612825) circle (0.075)
					(0.9238795325112867,0.3826834323650898) circle (0.075)
					(0.8314696123025452,0.5555702330196022) circle (0.075)
					(0.7071067811865476,0.7071067811865475) circle (0.075)
					(0.5555702330196023,0.8314696123025452) circle (0.075)
					(0.38268343236508984,0.9238795325112867) circle (0.075)
					(0.19509032201612833,0.9807852804032304) circle (0.075)
					(0,1) circle (0.075)
					(-0.1950903220161282,0.9807852804032304) circle (0.075)
					(-0.3826834323650897,0.9238795325112867) circle (0.075)
					(-0.555570233019602,0.8314696123025455) circle (0.075)
					(-0.7071067811865475,0.7071067811865476) circle (0.075)
					(-0.8314696123025453,0.5555702330196022) circle (0.075)
					(-0.9238795325112867,0.3826834323650899) circle (0.075)
					(-0.9807852804032304,0.1950903220161286) circle (0.075)
					(-1,0) circle (0.075)
					(-0.9807852804032304,-0.19509032201612836) circle (0.075)
					(-0.9238795325112868,-0.38268343236508967) circle (0.075)
					(-0.8314696123025455,-0.555570233019602) circle (0.075)
					(-0.7071067811865477,-0.7071067811865475) circle (0.075)
					(-0.5555702330196022,-0.8314696123025452) circle (0.075)
					(-0.38268343236509034,-0.9238795325112865) circle (0.075)
					(-0.19509032201612866,-0.9807852804032303) circle (0.075)
					(0,-1) circle (0.075)
					(0.1950903220161283,-0.9807852804032304) circle (0.075)
					(0.38268343236509,-0.9238795325112866) circle (0.075)
					(0.5555702330196018,-0.8314696123025455) circle (0.075)
					(0.7071067811865474,-0.7071067811865477) circle (0.075)
					(0.8314696123025452,-0.5555702330196022) circle (0.075)
					(0.9238795325112865,-0.3826834323650904) circle (0.075)
					(0.9807852804032303,-0.19509032201612872) circle (0.075)
				;
				\draw[line width = 0.7mm]
					\domainboundarycoordinates
				;
			\end{scope}

			\draw[dashed]
				(0,0) circle (0.8)
			;

			\draw[dotted]
				(-0.4,0) circle (0.2)
			;

			\node at (0,0.87) {$\Gamma_\text{int}$};
			\node at (0,-0.83) [node font = \LARGE] {$\mcV$};
			\node at (-0.4,0) {$D$};
		\end{tikzpicture}

		\caption{Configuration of the domain $\Omega$ for the first numerical test. The thicker parts on the boundary is $\Gamma_0$. The artificial boundary is the dashed line, while the dotted line represents the boundary of the support of the perturbation.}
		\label{fig:configuration_disk}
	\end{figure}

	We fix the physical constants $\omega$, $\mu_0$ and $\eps_0$ to 1. The background permittivity and conductivity are also set to 1. Then, the background refractive index is here $\kappa_0 \coloneqq 1 + i$. The refractive index that we will want to retrieve is a perturbation of this background index:
	\[\kappa_\text{ex} \coloneqq \kappa_0(1 + a_\text{ex}\chi_{D_\text{ex}}),\]
	where the support $D_\text{ex}$ is the disk centered at $\bfx_\text{ex} \coloneqq (-0.4,0)$ and of radius $r_\text{ex} \coloneqq 0.2$, and $a_\text{ex} \coloneqq 0.1$ is the amplitude. We find that the exact projection of the perturbation on $\Gamma_\text{int}$ is $\hat{\bfx}_\text{ex} = (-0.8,0)$. Then, the exact depth is $d_\text{ex} = 0.4$.

	The synthetic data are generated on a mesh of the unit disk $\Omega$ of size \scnum{0.0137998}, using \num{74525} triangles (\num{37584} vertices and \num{112108} edges). We use eight incident waves, of directions $\bfeta_m \coloneqq (\cos\theta_m,\sin\theta_m)$ with $\theta_m \coloneqq 2\pi\frac{m}{8}$ for $0 \leq m \leq 7$. The data transmission problem is solved on a mesh of the known neighborhood $\mcV$ of size \scnum{0.0200013}, using \num{18190} triangles (\num{9466} vertices and \num{27656} edges). The cost function $j$ is minimized on a mesh of the unit disk of size \scnum{0.0408549}, using \num{8691} triangles (\num{4459} vertices and \num{13149} edges). In \autoref{tab:2d_unit}, we summarize the results obtained in this configuration. The reconstructed refractive index is shown in \autoref{fig:reconstruction_2d_unit}.

	\begin{table}[hbt]
		\centering

		\begin{tabular}{c|ccc}
			\hline
			Parameter & Exact value & Approximation & Relative error \\
			\hline
			Center & (\num{-0.4}, \num{0}) & (\rounded[5]{-0.3723636117}, \rounded[5]{0.004823660273}) & \scnum[5]{0.07013547857083893} \\
			Radius & \num{0.2} & \rounded[5]{0.2115693343} & \scnum[5]{0.05784667149999992} \\
			Amplitude & \num{0.1} & \rounded[5]{0.09142762072} & \scnum[5]{0.0857237928} \\
			\hline
		\end{tabular}

		\caption{Reconstruction of a spherical perturbation in the unit disk in 2D, with unitary physical parameters.}
		\label{tab:2d_unit}
	\end{table}

	\begin{figure}[hbt]
		\centering
		\includegraphics[width=0.5\textwidth]{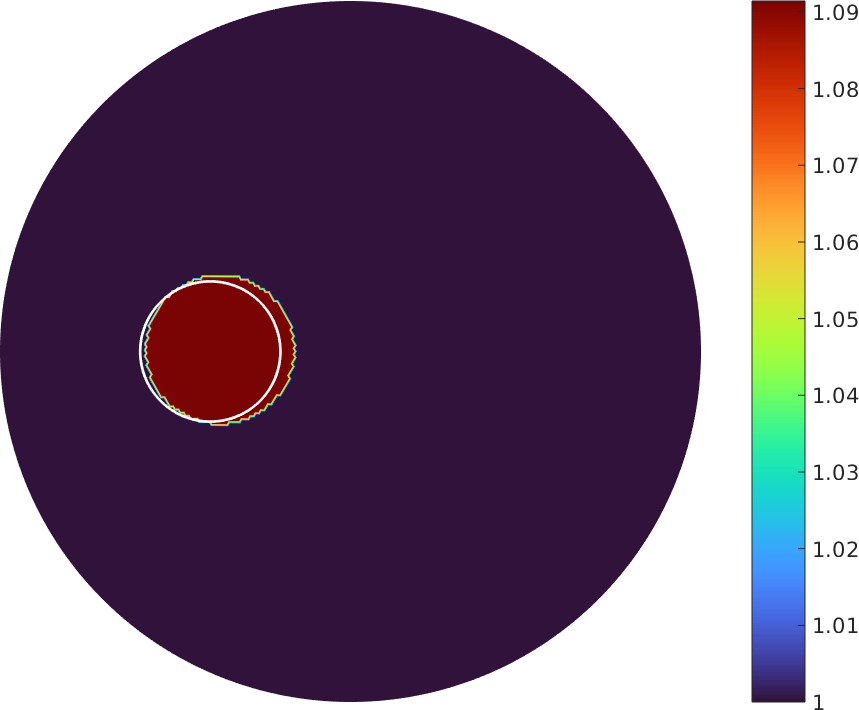}
		\caption{Real part of the reconstructed refractive index in the unit disk. The boundary of the exact support of the perturbation is shown as a white line.}
		\label{fig:reconstruction_2d_unit}
	\end{figure}

	\subsection{Influence of the amplitude}

	To determine the influence of the choice of amplitude over the results, we propose to test the same configuration as described in the previous section, but with different values of exact amplitude. We vary here the amplitude by taking values between \num{-0.3} and \num{0.3}, with a step of \num{0.05}, excluding zero. It is interesting to notice that the reconstruction of the support does not seem to be affected by the choice of amplitude. Indeed, the mean value of the reconstructed depths is \rounded[5]{0.42781628786666664}, with a standard variation of \rounded[5]{0.0003676384507637249}, and the mean value of the reconstructed radii is \rounded[5]{0.2132786510916667}, with a standard variation of \rounded[5]{0.0017723571690481182}. In \autoref{tab:different_amplitudes}, we list the values of the reconstructed amplitude for each case. We can observe that the error is always of same order: the mean value for this error is \rounded[2]{9.451839231111116}\%, with a standard variation of \rounded[2]{1.1448245492810216}\%. An interesting detail to notice is that the difference $a - a_\text{ex}$ is always of sign opposite to the sign of the exact amplitude. In other words, the absolute value of the amplitude is always under-estimated.

	\begin{table}[hbt]
		\centering

		\begin{tabular}{rrrr}
			\hline
			$a_\text{ex}$ & $a$ & $a - a_\text{ex}$ & $\frac{\abs{a - a_\text{ex}}}{\abs{a_\text{ex}}}$ \\
			\hline
			\num{-0.30} & \rounded[5]{-0.2695219108} & \rounded[5]{0.030478089200000003} & \rounded[5]{10.159363066666668}\% \\
			\num{-0.25} & \rounded[5]{-0.227478294} & \rounded[5]{0.022521706000000002} & \rounded[5]{9.008682400000001}\% \\
			\num{-0.20} & \rounded[5]{-0.1819287314} & \rounded[5]{0.018071268600000023} & \rounded[5]{9.035634300000012}\% \\
			\num{-0.15} & \rounded[5]{-0.1349982906} & \rounded[5]{0.015001709400000007} & \rounded[5]{10.001139600000005}\% \\
			\num{-0.10} & \rounded[5]{-0.09256767694} & \rounded[5]{0.0074323230600000095} & \rounded[5]{7.43232306000001}\% \\
			\num{-0.05} & \rounded[5]{-0.04519009318} & \rounded[5]{0.004809906820000001} & \rounded[5]{9.619813640000002}\% \\
			\num{0.05} & \rounded[5]{0.04445695574} & \rounded[5]{-0.005543044260000005} & \rounded[5]{11.08608852000001}\% \\
			\num{0.10} & \rounded[5]{0.09142825286} & \rounded[5]{-0.008571747140000002} & \rounded[5]{8.571747140000003}\% \\
			\num{0.15} & \rounded[5]{0.1325225288} & \rounded[5]{-0.017477471199999983} & \rounded[5]{11.651647466666656}\% \\
			\num{0.20} & \rounded[5]{0.1824419814} & \rounded[5]{-0.017558018600000003} & \rounded[5]{8.779009300000002}\% \\
			\num{0.25} & \rounded[5]{0.2253353438} & \rounded[5]{-0.024664656199999996} & \rounded[5]{9.865862479999999}\% \\
			\num{0.30} & \rounded[5]{0.2753677206} & \rounded[5]{-0.024632279400000012} & \rounded[5]{8.210759800000005}\% \\
			\hline
		\end{tabular}

		\caption{Behavior of the reconstructed amplitude $a$ with respect to the exact amplitude $a_\text{ex}$.}
		\label{tab:different_amplitudes}
	\end{table}

	\subsection{3D case}
	\label{sec:numerics_unit_ball}

	We now propose a 3D equivalent of the previous case. The domain $\Omega$ is the unit ball. The known neighborhood $\mcV$ is the annulus of outer radius 1 and inner radius $0.8$. The accessible part $\Gamma_0$ is the union of 102 patches distributed on the whole boundary $\Gamma$. This is illustrated in \autoref{fig:configuration_3d}. Inside $\mcV$, we define the artificial boundary $\Gamma_\text{int}$ as the sphere centered at the origin and of radius $0.9$. The perturbation is located in the ball centered at point $(-0.4,0,0)$ and of radius $0.2$.

	\begin{figure}[hbt]
		\centering

		\begin{tikzpicture}[scale = 3, 3d view = {0}{0}]
			\begin{scope}[
				plane origin = {(0,0,-0.98875)},
				plane x = {(0,1,-0.98875)},
				plane y = {(1,0,-0.98875)},
				canvas is plane
			]
				\draw[fill = gray!20] (0,0) circle (0.14957753006384342);
			\end{scope}
			\begin{scope}[
				plane origin = {(0.3381724167,0,-0.9291210788)},
				plane x = {(0.3381724167,1,-0.9291210788)},
				plane y = {(1.2778650375,0,-0.5871009355)},
				canvas is plane
			]
				\draw[fill = gray!20] (0,0) circle (0.14957753006384342);
			\end{scope}
			\begin{scope}[
				plane origin = {(0.1045010238,0.3216210805,-0.9291210788)},
				plane x = {(0.068144988,1.2672386594,-0.6058787034)},
				plane y = {(1.0982353217,0.3216210805,-0.8173528132)},
				canvas is plane
			]
				\draw[fill = gray!20] (0,0) circle (0.14957753006384342);
			\end{scope}
			\begin{scope}[
				plane origin = {(-0.2735872322,0.1987727592,-0.9291210788)},
				plane x = {(-0.2168016693,1.1783569422,-0.7362733972)},
				plane y = {(0.6856898132,0.1987727592,-1.2115879798)},
				canvas is plane
			]
				\draw[fill = gray!20] (0,0) circle (0.14957753006384342);
			\end{scope}
			\begin{scope}[
				plane origin = {(-0.2735872322,-0.1987727592,-0.9291210788)},
				plane x = {(-0.3303727951,0.7808114237,-1.1219687604)},
				plane y = {(0.6856898132,-0.1987727592,-1.2115879798)},
				canvas is plane
			]
				\draw[fill = gray!20] (0,0) circle (0.14957753006384342);
			\end{scope}
			\begin{scope}[
				plane origin = {(0.1045010238,-0.3216210805,-0.9291210788)},
				plane x = {(0.1408570595,0.6239964983,-1.2523634542)},
				plane y = {(1.0982353217,-0.3216210805,-0.8173528132)},
				canvas is plane
			]
				\draw[fill = gray!20] (0,0) circle (0.14957753006384342);
			\end{scope}
			\begin{scope}[
				plane origin = {(0.6355562491,0,-0.7574264431)},
				plane x = {(0.6355562491,1,-0.7574264431)},
				plane y = {(1.4016006922,0,-0.1146388334)},
				canvas is plane
			]
				\draw[fill = gray!20] (0,0) circle (0.14957753006384342);
			\end{scope}
			\begin{scope}[
				plane origin = {(0.5346639398,0.3436076501,-0.7574264431)},
				plane x = {(0.3342541021,1.2812812514,-0.4735178059)},
				plane y = {(1.3516267316,0.3436076501,-0.1807360609)},
				canvas is plane
			]
				\draw[fill = gray!20] (0,0) circle (0.14957753006384342);
			\end{scope}
			\begin{scope}[
				plane origin = {(0.2640196074,0.578122299,-0.7574264431)},
				plane x = {(0.0715649248,1.38937177,-0.2053073519)},
				plane y = {(1.2082968865,0.578122299,-0.4282753833)},
				canvas is plane
			]
				\draw[fill = gray!20] (0,0) circle (0.14957753006384342);
			\end{scope}
			\begin{scope}[
				plane origin = {(-0.0904490848,0.6290872028,-0.7574264432)},
				plane x = {(-0.0150070553,1.4005743075,-0.1256700443)},
				plane y = {(0.9024961485,0.6290872028,-0.8760003171)},
				canvas is plane
			]
				\draw[fill = gray!20] (0,0) circle (0.14957753006384342);
			\end{scope}
			\begin{scope}[
				plane origin = {(-0.4162008318,0.4803213647,-0.7574264431)},
				plane x = {(-0.1822568358,1.3543988858,-0.3316815737)},
				plane y = {(0.4602024901,0.4803213647,-1.2390042846)},
				canvas is plane
			]
				\draw[fill = gray!20] (0,0) circle (0.14957753006384342);
			\end{scope}
			\begin{scope}[
				plane origin = {(-0.6098117553,0.1790568871,-0.7574264431)},
				plane x = {(-0.496244119,1.1625226424,-0.6163679443)},
				plane y = {(0.1691115971,0.1790568871,-1.3845455794)},
				canvas is plane
			]
				\draw[fill = gray!20] (0,0) circle (0.14957753006384342);
			\end{scope}
			\begin{scope}[
				plane origin = {(-0.6098117553,-0.1790568871,-0.7574264431)},
				plane x = {(-0.7233793916,0.8044088682,-0.898484942)},
				plane y = {(0.1691115971,-0.1790568871,-1.3845455794)},
				canvas is plane
			]
				\draw[fill = gray!20] (0,0) circle (0.14957753006384342);
			\end{scope}
			\begin{scope}[
				plane origin = {(-0.4162008318,-0.4803213647,-0.7574264431)},
				plane x = {(-0.6501448278,0.3937561564,-1.1831713125)},
				plane y = {(0.4602024901,-0.4803213647,-1.2390042846)},
				canvas is plane
			]
				\draw[fill = gray!20] (0,0) circle (0.14957753006384342);
			\end{scope}
			\begin{scope}[
				plane origin = {(-0.0904490848,-0.6290872028,-0.7574264432)},
				plane x = {(-0.1658911143,0.1423999018,-1.389182842)},
				plane y = {(0.9024961485,-0.6290872028,-0.8760003171)},
				canvas is plane
			]
				\draw[fill = gray!20] (0,0) circle (0.14957753006384342);
			\end{scope}
			\begin{scope}[
				plane origin = {(0.2640196074,-0.578122299,-0.7574264431)},
				plane x = {(0.4564742901,0.233127172,-1.3095455344)},
				plane y = {(1.2082968865,-0.578122299,-0.4282753833)},
				canvas is plane
			]
				\draw[fill = gray!20] (0,0) circle (0.14957753006384342);
			\end{scope}
			\begin{scope}[
				plane origin = {(0.5346639398,-0.3436076501,-0.7574264431)},
				plane x = {(0.7350737776,0.5940659512,-1.0413350804)},
				plane y = {(1.3516267316,-0.3436076501,-0.1807360609)},
				canvas is plane
			]
				\draw[fill = gray!20] (0,0) circle (0.14957753006384342);
			\end{scope}
			\begin{scope}[
				plane origin = {(0.856282618,0,-0.494375)},
				plane x = {(0.856282618,1,-0.494375)},
				plane y = {(1.356282618,0,0.3716504038)},
				canvas is plane
			]
				\draw[fill = gray!20] (0,0) circle (0.14957753006384342);
			\end{scope}
			\begin{scope}[
				plane origin = {(0.7911019848,0.3276851713,-0.494375)},
				plane x = {(0.5100536593,1.271170753,-0.3187424411)},
				plane y = {(1.3210517898,0.3276851713,0.3536540114)},
				canvas is plane
			]
				\draw[fill = gray!20] (0,0) circle (0.14957753006384342);
			\end{scope}
			\begin{scope}[
				plane origin = {(0.6054832458,0.6054832458,-0.494375)},
				plane x = {(0.1311415968,1.3960526608,-0.1070766654)},
				plane y = {(1.2379387778,0.6054832458,0.2802216692)},
				canvas is plane
			]
				\draw[fill = gray!20] (0,0) circle (0.14957753006384342);
			\end{scope}
			\begin{scope}[
				plane origin = {(0.3276851713,0.7911019848,-0.494375)},
				plane x = {(-0.1143579064,1.3909644333,0.1725305108)},
				plane y = {(1.1612095923,0.7911019848,0.0581076147)},
				canvas is plane
			]
				\draw[fill = gray!20] (0,0) circle (0.14957753006384342);
			\end{scope}
			\begin{scope}[
				plane origin = {(0,0.856282618,-0.494375)},
				plane x = {(0,1.356282618,0.3716504038)},
				plane y = {(1,0.856282618,-0.494375)},
				canvas is plane
			]
				\draw[fill = gray!20] (0,0) circle (0.14957753006384342);
			\end{scope}
			\begin{scope}[
				plane origin = {(-0.3276851713,0.7911019848,-0.494375)},
				plane x = {(0.1143579064,1.3909644333,0.1725305108)},
				plane y = {(0.5058392497,0.7911019848,-1.0468576148)},
				canvas is plane
			]
				\draw[fill = gray!20] (0,0) circle (0.14957753006384342);
			\end{scope}
			\begin{scope}[
				plane origin = {(-0.6054832458,0.6054832458,-0.494375)},
				plane x = {(-0.1311415968,1.3960526608,-0.1070766654)},
				plane y = {(0.0269722862,0.6054832458,-1.2689716692)},
				canvas is plane
			]
				\draw[fill = gray!20] (0,0) circle (0.14957753006384342);
			\end{scope}
			\begin{scope}[
				plane origin = {(-0.7911019848,0.3276851713,-0.494375)},
				plane x = {(-0.5100536593,1.271170753,-0.3187424411)},
				plane y = {(-0.2611521799,0.3276851713,-1.3424040114)},
				canvas is plane
			]
				\draw[fill = gray!20] (0,0) circle (0.14957753006384342);
			\end{scope}
			\begin{scope}[
				plane origin = {(-0.856282618,0,-0.494375)},
				plane x = {(-0.856282618,1,-0.494375)},
				plane y = {(-0.356282618,0,-1.3604004038)},
				canvas is plane
			]
				\draw[fill = gray!20] (0,0) circle (0.14957753006384342);
			\end{scope}
			\begin{scope}[
				plane origin = {(-0.7911019848,-0.3276851713,-0.494375)},
				plane x = {(-1.0721503104,0.6158004105,-0.6700075589)},
				plane y = {(-0.2611521799,-0.3276851713,-1.3424040114)},
				canvas is plane
			]
				\draw[fill = gray!20] (0,0) circle (0.14957753006384342);
			\end{scope}
			\begin{scope}[
				plane origin = {(-0.6054832458,-0.6054832458,-0.494375)},
				plane x = {(-1.0798248948,0.1850861692,-0.8816733346)},
				plane y = {(0.0269722862,-0.6054832458,-1.2689716692)},
				canvas is plane
			]
				\draw[fill = gray!20] (0,0) circle (0.14957753006384342);
			\end{scope}
			\begin{scope}[
				plane origin = {(-0.3276851713,-0.7911019848,-0.494375)},
				plane x = {(-0.769728249,-0.1912395364,-1.1612805108)},
				plane y = {(0.5058392497,-0.7911019848,-1.0468576148)},
				canvas is plane
			]
				\draw[fill = gray!20] (0,0) circle (0.14957753006384342);
			\end{scope}
			\begin{scope}[
				plane origin = {(0,-0.856282618,-0.494375)},
				plane x = {(0,-0.356282618,-1.3604004038)},
				plane y = {(1,-0.856282618,-0.494375)},
				canvas is plane
			]
				\draw[fill = gray!20] (0,0) circle (0.14957753006384342);
			\end{scope}
			\begin{scope}[
				plane origin = {(0.3276851713,-0.7911019848,-0.494375)},
				plane x = {(0.769728249,-0.1912395364,-1.1612805108)},
				plane y = {(1.1612095923,-0.7911019848,0.0581076147)},
				canvas is plane
			]
				\draw[fill = gray!20] (0,0) circle (0.14957753006384342);
			\end{scope}
			\begin{scope}[
				plane origin = {(0.6054832458,-0.6054832458,-0.494375)},
				plane x = {(1.0798248948,0.1850861692,-0.8816733346)},
				plane y = {(1.2379387778,-0.6054832458,0.2802216692)},
				canvas is plane
			]
				\draw[fill = gray!20] (0,0) circle (0.14957753006384342);
			\end{scope}
			\begin{scope}[
				plane origin = {(0.7911019848,-0.3276851713,-0.494375)},
				plane x = {(1.0721503104,0.6158004105,-0.6700075589)},
				plane y = {(1.3210517898,-0.3276851713,0.3536540114)},
				canvas is plane
			]
				\draw[fill = gray!20] (0,0) circle (0.14957753006384342);
			\end{scope}
			\begin{scope}[
				plane origin = {(0.9737286658,0,-0.1716946357)},
				plane x = {(0.9737286658,1,-0.1716946357)},
				plane y = {(1.1473768435,0,0.8131131173)},
				canvas is plane
			]
				\draw[fill = gray!20] (0,0) circle (0.14957753006384342);
			\end{scope}
			\begin{scope}[
				plane origin = {(0.9150056419,0.3330348178,-0.1716946357)},
				plane x = {(0.5839592062,1.2746024113,-0.1095760054)},
				plane y = {(1.0994301897,0.3330348178,0.8111520391)},
				canvas is plane
			]
				\draw[fill = gray!20] (0,0) circle (0.14957753006384342);
			\end{scope}
			\begin{scope}[
				plane origin = {(0.7459194335,0.6259007216,-0.1716946357)},
				plane x = {(0.1290284106,1.4000343451,-0.0296995694)},
				plane y = {(0.9702323589,0.6259007216,0.8028225327)},
				canvas is plane
			]
				\draw[fill = gray!20] (0,0) circle (0.14957753006384342);
			\end{scope}
			\begin{scope}[
				plane origin = {(0.4868643328,0.843273761,-0.1716946357)},
				plane x = {(-0.3174549248,1.3653994807,0.1119517368)},
				plane y = {(0.8194435779,0.843273761,0.7713806772)},
				canvas is plane
			]
				\draw[fill = gray!20] (0,0) circle (0.14957753006384342);
			\end{scope}
			\begin{scope}[
				plane origin = {(0.1690862083,0.9589355394,-0.1716946356)},
				plane x = {(-0.5114295796,1.2026528638,0.5193192052)},
				plane y = {(0.8815845015,0.9589355394,0.5299792075)},
				canvas is plane
			]
				\draw[fill = gray!20] (0,0) circle (0.14957753006384342);
			\end{scope}
			\begin{scope}[
				plane origin = {(-0.1690862083,0.9589355394,-0.1716946356)},
				plane x = {(0.5114295796,1.2026528638,0.5193192052)},
				plane y = {(0.5434120849,0.9589355394,-0.8733684788)},
				canvas is plane
			]
				\draw[fill = gray!20] (0,0) circle (0.14957753006384342);
			\end{scope}
			\begin{scope}[
				plane origin = {(-0.4868643328,0.843273761,-0.1716946357)},
				plane x = {(0.3174549248,1.3653994807,0.1119517368)},
				plane y = {(-0.1542850877,0.843273761,-1.1147699485)},
				canvas is plane
			]
				\draw[fill = gray!20] (0,0) circle (0.14957753006384342);
			\end{scope}
			\begin{scope}[
				plane origin = {(-0.7459194335,0.6259007216,-0.1716946357)},
				plane x = {(-0.1290284106,1.4000343451,-0.0296995694)},
				plane y = {(-0.5216065081,0.6259007216,-1.1462118041)},
				canvas is plane
			]
				\draw[fill = gray!20] (0,0) circle (0.14957753006384342);
			\end{scope}
			\begin{scope}[
				plane origin = {(-0.9150056419,0.3330348178,-0.1716946357)},
				plane x = {(-0.5839592062,1.2746024113,-0.1095760054)},
				plane y = {(-0.7305810941,0.3330348178,-1.1545413105)},
				canvas is plane
			]
				\draw[fill = gray!20] (0,0) circle (0.14957753006384342);
			\end{scope}
			\begin{scope}[
				plane origin = {(-0.9737286658,0,-0.1716946357)},
				plane x = {(-0.9737286658,1,-0.1716946357)},
				plane y = {(-0.8000804881,0,-1.1565023887)},
				canvas is plane
			]
				\draw[fill = gray!20] (0,0) circle (0.14957753006384342);
			\end{scope}
			\begin{scope}[
				plane origin = {(-0.9150056419,-0.3330348178,-0.1716946357)},
				plane x = {(-1.2460520776,0.6085327757,-0.233813266)},
				plane y = {(-0.7305810941,-0.3330348178,-1.1545413105)},
				canvas is plane
			]
				\draw[fill = gray!20] (0,0) circle (0.14957753006384342);
			\end{scope}
			\begin{scope}[
				plane origin = {(-0.7459194335,-0.6259007216,-0.1716946357)},
				plane x = {(-1.3628104564,0.1482329018,-0.313689702)},
				plane y = {(-0.5216065081,-0.6259007216,-1.1462118041)},
				canvas is plane
			]
				\draw[fill = gray!20] (0,0) circle (0.14957753006384342);
			\end{scope}
			\begin{scope}[
				plane origin = {(-0.4868643328,-0.843273761,-0.1716946357)},
				plane x = {(-1.2911835905,-0.3211480414,-0.4553410082)},
				plane y = {(-0.1542850877,-0.843273761,-1.1147699485)},
				canvas is plane
			]
				\draw[fill = gray!20] (0,0) circle (0.14957753006384342);
			\end{scope}
			\begin{scope}[
				plane origin = {(-0.1690862083,-0.9589355394,-0.1716946356)},
				plane x = {(-0.8496019962,-0.715218215,-0.8627084765)},
				plane y = {(0.5434120849,-0.9589355394,-0.8733684788)},
				canvas is plane
			]
				\draw[fill = gray!20] (0,0) circle (0.14957753006384342);
			\end{scope}
			\begin{scope}[
				plane origin = {(0.1690862083,-0.9589355394,-0.1716946356)},
				plane x = {(0.8496019962,-0.715218215,-0.8627084765)},
				plane y = {(0.8815845015,-0.9589355394,0.5299792075)},
				canvas is plane
			]
				\draw[fill = gray!20] (0,0) circle (0.14957753006384342);
			\end{scope}
			\begin{scope}[
				plane origin = {(0.4868643328,-0.843273761,-0.1716946357)},
				plane x = {(1.2911835905,-0.3211480414,-0.4553410082)},
				plane y = {(0.8194435779,-0.843273761,0.7713806772)},
				canvas is plane
			]
				\draw[fill = gray!20] (0,0) circle (0.14957753006384342);
			\end{scope}
			\begin{scope}[
				plane origin = {(0.7459194335,-0.6259007216,-0.1716946357)},
				plane x = {(1.3628104564,0.1482329018,-0.313689702)},
				plane y = {(0.9702323589,-0.6259007216,0.8028225327)},
				canvas is plane
			]
				\draw[fill = gray!20] (0,0) circle (0.14957753006384342);
			\end{scope}
			\begin{scope}[
				plane origin = {(0.9150056419,-0.3330348178,-0.1716946357)},
				plane x = {(1.2460520776,0.6085327757,-0.233813266)},
				plane y = {(1.0994301897,-0.3330348178,0.8111520391)},
				canvas is plane
			]
				\draw[fill = gray!20] (0,0) circle (0.14957753006384342);
			\end{scope}
			\begin{scope}[
				plane origin = {(0.9737286658,0,0.1716946357)},
				plane x = {(0.9737286658,1,0.1716946357)},
				plane y = {(0.8000804881,0,1.1565023887)},
				canvas is plane
			]
				\draw[fill = gray!20] (0,0) circle (0.14957753006384342);
			\end{scope}
			\begin{scope}[
				plane origin = {(0.9150056419,0.3330348178,0.1716946357)},
				plane x = {(0.5839592062,1.2746024113,0.1095760054)},
				plane y = {(0.7305810941,0.3330348178,1.1545413105)},
				canvas is plane
			]
				\draw[fill = gray!20] (0,0) circle (0.14957753006384342);
			\end{scope}
			\begin{scope}[
				plane origin = {(0.7459194335,0.6259007216,0.1716946357)},
				plane x = {(0.1290284106,1.4000343451,0.0296995694)},
				plane y = {(0.5216065081,0.6259007216,1.1462118041)},
				canvas is plane
			]
				\draw[fill = gray!20] (0,0) circle (0.14957753006384342);
			\end{scope}
			\begin{scope}[
				plane origin = {(0.4868643328,0.843273761,0.1716946357)},
				plane x = {(-0.3174549248,1.3653994807,-0.1119517368)},
				plane y = {(0.1542850877,0.843273761,1.1147699485)},
				canvas is plane
			]
				\draw[fill = gray!20] (0,0) circle (0.14957753006384342);
			\end{scope}
			\begin{scope}[
				plane origin = {(0.1690862083,0.9589355394,0.1716946356)},
				plane x = {(-0.5114295796,1.2026528638,-0.5193192052)},
				plane y = {(-0.5434120849,0.9589355394,0.8733684788)},
				canvas is plane
			]
				\draw[fill = gray!20] (0,0) circle (0.14957753006384342);
			\end{scope}
			\begin{scope}[
				plane origin = {(-0.1690862083,0.9589355394,0.1716946356)},
				plane x = {(0.5114295796,1.2026528638,-0.5193192052)},
				plane y = {(-0.8815845015,0.9589355394,-0.5299792075)},
				canvas is plane
			]
				\draw[fill = gray!20] (0,0) circle (0.14957753006384342);
			\end{scope}
			\begin{scope}[
				plane origin = {(-0.4868643328,0.843273761,0.1716946357)},
				plane x = {(0.3174549248,1.3653994807,-0.1119517368)},
				plane y = {(-0.8194435779,0.843273761,-0.7713806772)},
				canvas is plane
			]
				\draw[fill = gray!20] (0,0) circle (0.14957753006384342);
			\end{scope}
			\begin{scope}[
				plane origin = {(-0.7459194335,0.6259007216,0.1716946357)},
				plane x = {(-0.1290284106,1.4000343451,0.0296995694)},
				plane y = {(-0.9702323589,0.6259007216,-0.8028225327)},
				canvas is plane
			]
				\draw[fill = gray!20] (0,0) circle (0.14957753006384342);
			\end{scope}
			\begin{scope}[
				plane origin = {(-0.9150056419,0.3330348178,0.1716946357)},
				plane x = {(-0.5839592062,1.2746024113,0.1095760054)},
				plane y = {(-1.0994301897,0.3330348178,-0.8111520391)},
				canvas is plane
			]
				\draw[fill = gray!20] (0,0) circle (0.14957753006384342);
			\end{scope}
			\begin{scope}[
				plane origin = {(-0.9737286658,0,0.1716946357)},
				plane x = {(-0.9737286658,1,0.1716946357)},
				plane y = {(-1.1473768435,0,-0.8131131173)},
				canvas is plane
			]
				\draw[fill = gray!20] (0,0) circle (0.14957753006384342);
			\end{scope}
			\begin{scope}[
				plane origin = {(-0.9150056419,-0.3330348178,0.1716946357)},
				plane x = {(-1.2460520776,0.6085327757,0.233813266)},
				plane y = {(-1.0994301897,-0.3330348178,-0.8111520391)},
				canvas is plane
			]
				\draw[fill = gray!20] (0,0) circle (0.14957753006384342);
			\end{scope}
			\begin{scope}[
				plane origin = {(-0.7459194335,-0.6259007216,0.1716946357)},
				plane x = {(-1.3628104564,0.1482329018,0.313689702)},
				plane y = {(-0.9702323589,-0.6259007216,-0.8028225327)},
				canvas is plane
			]
				\draw[fill = gray!20] (0,0) circle (0.14957753006384342);
			\end{scope}
			\begin{scope}[
				plane origin = {(-0.4868643328,-0.843273761,0.1716946357)},
				plane x = {(-1.2911835905,-0.3211480414,0.4553410082)},
				plane y = {(-0.8194435779,-0.843273761,-0.7713806772)},
				canvas is plane
			]
				\draw[fill = gray!20] (0,0) circle (0.14957753006384342);
			\end{scope}
			\begin{scope}[
				plane origin = {(-0.1690862083,-0.9589355394,0.1716946356)},
				plane x = {(-0.8496019962,-0.715218215,0.8627084765)},
				plane y = {(-0.8815845015,-0.9589355394,-0.5299792075)},
				canvas is plane
			]
				\draw[fill = gray!20] (0,0) circle (0.14957753006384342);
			\end{scope}
			\begin{scope}[
				plane origin = {(0.1690862083,-0.9589355394,0.1716946356)},
				plane x = {(0.8496019962,-0.715218215,0.8627084765)},
				plane y = {(-0.5434120849,-0.9589355394,0.8733684788)},
				canvas is plane
			]
				\draw[fill = gray!20] (0,0) circle (0.14957753006384342);
			\end{scope}
			\begin{scope}[
				plane origin = {(0.4868643328,-0.843273761,0.1716946357)},
				plane x = {(1.2911835905,-0.3211480414,0.4553410082)},
				plane y = {(0.1542850877,-0.843273761,1.1147699485)},
				canvas is plane
			]
				\draw[fill = gray!20] (0,0) circle (0.14957753006384342);
			\end{scope}
			\begin{scope}[
				plane origin = {(0.7459194335,-0.6259007216,0.1716946357)},
				plane x = {(1.3628104564,0.1482329018,0.313689702)},
				plane y = {(0.5216065081,-0.6259007216,1.1462118041)},
				canvas is plane
			]
				\draw[fill = gray!20] (0,0) circle (0.14957753006384342);
			\end{scope}
			\begin{scope}[
				plane origin = {(0.9150056419,-0.3330348178,0.1716946357)},
				plane x = {(1.2460520776,0.6085327757,0.233813266)},
				plane y = {(0.7305810941,-0.3330348178,1.1545413105)},
				canvas is plane
			]
				\draw[fill = gray!20] (0,0) circle (0.14957753006384342);
			\end{scope}
			\begin{scope}[
				plane origin = {(0.856282618,0,0.494375)},
				plane x = {(0.856282618,1,0.494375)},
				plane y = {(0.356282618,0,1.3604004038)},
				canvas is plane
			]
				\draw[fill = gray!20] (0,0) circle (0.14957753006384342);
			\end{scope}
			\begin{scope}[
				plane origin = {(0.7911019848,0.3276851713,0.494375)},
				plane x = {(0.5100536593,1.271170753,0.3187424411)},
				plane y = {(0.2611521799,0.3276851713,1.3424040114)},
				canvas is plane
			]
				\draw[fill = gray!20] (0,0) circle (0.14957753006384342);
			\end{scope}
			\begin{scope}[
				plane origin = {(0.6054832458,0.6054832458,0.494375)},
				plane x = {(0.1311415968,1.3960526608,0.1070766654)},
				plane y = {(-0.0269722862,0.6054832458,1.2689716692)},
				canvas is plane
			]
				\draw[fill = gray!20] (0,0) circle (0.14957753006384342);
			\end{scope}
			\begin{scope}[
				plane origin = {(0.3276851713,0.7911019848,0.494375)},
				plane x = {(-0.1143579064,1.3909644333,-0.1725305108)},
				plane y = {(-0.5058392497,0.7911019848,1.0468576148)},
				canvas is plane
			]
				\draw[fill = gray!20] (0,0) circle (0.14957753006384342);
			\end{scope}
			\begin{scope}[
				plane origin = {(0,0.856282618,0.494375)},
				plane x = {(0,1.356282618,-0.3716504038)},
				plane y = {(-1,0.856282618,0.494375)},
				canvas is plane
			]
				\draw[fill = gray!20] (0,0) circle (0.14957753006384342);
			\end{scope}
			\begin{scope}[
				plane origin = {(-0.3276851713,0.7911019848,0.494375)},
				plane x = {(0.1143579064,1.3909644333,-0.1725305108)},
				plane y = {(-1.1612095923,0.7911019848,-0.0581076147)},
				canvas is plane
			]
				\draw[fill = gray!20] (0,0) circle (0.14957753006384342);
			\end{scope}
			\begin{scope}[
				plane origin = {(-0.6054832458,0.6054832458,0.494375)},
				plane x = {(-0.1311415968,1.3960526608,0.1070766654)},
				plane y = {(-1.2379387778,0.6054832458,-0.2802216692)},
				canvas is plane
			]
				\draw[fill = gray!20] (0,0) circle (0.14957753006384342);
			\end{scope}
			\begin{scope}[
				plane origin = {(-0.7911019848,0.3276851713,0.494375)},
				plane x = {(-0.5100536593,1.271170753,0.3187424411)},
				plane y = {(-1.3210517898,0.3276851713,-0.3536540114)},
				canvas is plane
			]
				\draw[fill = gray!20] (0,0) circle (0.14957753006384342);
			\end{scope}
			\begin{scope}[
				plane origin = {(-0.856282618,0,0.494375)},
				plane x = {(-0.856282618,1,0.494375)},
				plane y = {(-1.356282618,0,-0.3716504038)},
				canvas is plane
			]
				\draw[fill = gray!20] (0,0) circle (0.14957753006384342);
			\end{scope}
			\begin{scope}[
				plane origin = {(-0.7911019848,-0.3276851713,0.494375)},
				plane x = {(-1.0721503104,0.6158004105,0.6700075589)},
				plane y = {(-1.3210517898,-0.3276851713,-0.3536540114)},
				canvas is plane
			]
				\draw[fill = gray!20] (0,0) circle (0.14957753006384342);
			\end{scope}
			\begin{scope}[
				plane origin = {(-0.6054832458,-0.6054832458,0.494375)},
				plane x = {(-1.0798248948,0.1850861692,0.8816733346)},
				plane y = {(-1.2379387778,-0.6054832458,-0.2802216692)},
				canvas is plane
			]
				\draw[fill = gray!20] (0,0) circle (0.14957753006384342);
			\end{scope}
			\begin{scope}[
				plane origin = {(-0.3276851713,-0.7911019848,0.494375)},
				plane x = {(-0.769728249,-0.1912395364,1.1612805108)},
				plane y = {(-1.1612095923,-0.7911019848,-0.0581076147)},
				canvas is plane
			]
				\draw[fill = gray!20] (0,0) circle (0.14957753006384342);
			\end{scope}
			\begin{scope}[
				plane origin = {(0,-0.856282618,0.494375)},
				plane x = {(0,-0.356282618,1.3604004038)},
				plane y = {(-1,-0.856282618,0.494375)},
				canvas is plane
			]
				\draw[fill = gray!20] (0,0) circle (0.14957753006384342);
			\end{scope}
			\begin{scope}[
				plane origin = {(0.3276851713,-0.7911019848,0.494375)},
				plane x = {(0.769728249,-0.1912395364,1.1612805108)},
				plane y = {(-0.5058392497,-0.7911019848,1.0468576148)},
				canvas is plane
			]
				\draw[fill = gray!20] (0,0) circle (0.14957753006384342);
			\end{scope}
			\begin{scope}[
				plane origin = {(0.6054832458,-0.6054832458,0.494375)},
				plane x = {(1.0798248948,0.1850861692,0.8816733346)},
				plane y = {(-0.0269722862,-0.6054832458,1.2689716692)},
				canvas is plane
			]
				\draw[fill = gray!20] (0,0) circle (0.14957753006384342);
			\end{scope}
			\begin{scope}[
				plane origin = {(0.7911019848,-0.3276851713,0.494375)},
				plane x = {(1.0721503104,0.6158004105,0.6700075589)},
				plane y = {(0.2611521799,-0.3276851713,1.3424040114)},
				canvas is plane
			]
				\draw[fill = gray!20] (0,0) circle (0.14957753006384342);
			\end{scope}
			\begin{scope}[
				plane origin = {(0.6355562491,0,0.7574264431)},
				plane x = {(0.6355562491,1,0.7574264431)},
				plane y = {(-0.130488194,0,1.4002140528)},
				canvas is plane
			]
				\draw[fill = gray!20] (0,0) circle (0.14957753006384342);
			\end{scope}
			\begin{scope}[
				plane origin = {(0.5346639398,0.3436076501,0.7574264431)},
				plane x = {(0.3342541021,1.2812812514,0.4735178059)},
				plane y = {(-0.2822988519,0.3436076501,1.3341168253)},
				canvas is plane
			]
				\draw[fill = gray!20] (0,0) circle (0.14957753006384342);
			\end{scope}
			\begin{scope}[
				plane origin = {(0.2640196074,0.578122299,0.7574264431)},
				plane x = {(0.0715649248,1.38937177,0.2053073519)},
				plane y = {(-0.6802576716,0.578122299,1.086577503)},
				canvas is plane
			]
				\draw[fill = gray!20] (0,0) circle (0.14957753006384342);
			\end{scope}
			\begin{scope}[
				plane origin = {(-0.0904490848,0.6290872028,0.7574264432)},
				plane x = {(-0.0150070553,1.4005743075,0.1256700443)},
				plane y = {(-1.0833943181,0.6290872028,0.6388525692)},
				canvas is plane
			]
				\draw[fill = gray!20] (0,0) circle (0.14957753006384342);
			\end{scope}
			\begin{scope}[
				plane origin = {(-0.4162008318,0.4803213647,0.7574264431)},
				plane x = {(-0.1822568358,1.3543988858,0.3316815737)},
				plane y = {(-1.2926041537,0.4803213647,0.2758486017)},
				canvas is plane
			]
				\draw[fill = gray!20] (0,0) circle (0.14957753006384342);
			\end{scope}
			\begin{scope}[
				plane origin = {(-0.6098117553,0.1790568871,0.7574264431)},
				plane x = {(-0.496244119,1.1625226424,0.6163679443)},
				plane y = {(-1.3887351077,0.1790568871,0.1303073069)},
				canvas is plane
			]
				\draw[fill = gray!20] (0,0) circle (0.14957753006384342);
			\end{scope}
			\begin{scope}[
				plane origin = {(-0.6098117553,-0.1790568871,0.7574264431)},
				plane x = {(-0.7233793916,0.8044088682,0.898484942)},
				plane y = {(-1.3887351077,-0.1790568871,0.1303073069)},
				canvas is plane
			]
				\draw[fill = gray!20] (0,0) circle (0.14957753006384342);
			\end{scope}
			\begin{scope}[
				plane origin = {(-0.4162008318,-0.4803213647,0.7574264431)},
				plane x = {(-0.6501448278,0.3937561564,1.1831713125)},
				plane y = {(-1.2926041537,-0.4803213647,0.2758486017)},
				canvas is plane
			]
				\draw[fill = gray!20] (0,0) circle (0.14957753006384342);
			\end{scope}
			\begin{scope}[
				plane origin = {(-0.0904490848,-0.6290872028,0.7574264432)},
				plane x = {(-0.1658911143,0.1423999018,1.389182842)},
				plane y = {(-1.0833943181,-0.6290872028,0.6388525692)},
				canvas is plane
			]
				\draw[fill = gray!20] (0,0) circle (0.14957753006384342);
			\end{scope}
			\begin{scope}[
				plane origin = {(0.2640196074,-0.578122299,0.7574264431)},
				plane x = {(0.4564742901,0.233127172,1.3095455344)},
				plane y = {(-0.6802576716,-0.578122299,1.086577503)},
				canvas is plane
			]
				\draw[fill = gray!20] (0,0) circle (0.14957753006384342);
			\end{scope}
			\begin{scope}[
				plane origin = {(0.5346639398,-0.3436076501,0.7574264431)},
				plane x = {(0.7350737776,0.5940659512,1.0413350804)},
				plane y = {(-0.2822988519,-0.3436076501,1.3341168253)},
				canvas is plane
			]
				\draw[fill = gray!20] (0,0) circle (0.14957753006384342);
			\end{scope}
			\begin{scope}[
				plane origin = {(0.3381724167,0,0.9291210788)},
				plane x = {(0.3381724167,1,0.9291210788)},
				plane y = {(-0.6015202041,0,1.2711412221)},
				canvas is plane
			]
				\draw[fill = gray!20] (0,0) circle (0.14957753006384342);
			\end{scope}
			\begin{scope}[
				plane origin = {(0.1045010238,0.3216210805,0.9291210788)},
				plane x = {(0.068144988,1.2672386594,0.6058787034)},
				plane y = {(-0.8892332741,0.3216210805,1.0408893444)},
				canvas is plane
			]
				\draw[fill = gray!20] (0,0) circle (0.14957753006384342);
			\end{scope}
			\begin{scope}[
				plane origin = {(-0.2735872322,0.1987727592,0.9291210788)},
				plane x = {(-0.2168016693,1.1783569422,0.7362733972)},
				plane y = {(-1.2328642776,0.1987727592,0.6466541778)},
				canvas is plane
			]
				\draw[fill = gray!20] (0,0) circle (0.14957753006384342);
			\end{scope}
			\begin{scope}[
				plane origin = {(-0.2735872322,-0.1987727592,0.9291210788)},
				plane x = {(-0.3303727951,0.7808114237,1.1219687604)},
				plane y = {(-1.2328642776,-0.1987727592,0.6466541778)},
				canvas is plane
			]
				\draw[fill = gray!20] (0,0) circle (0.14957753006384342);
			\end{scope}
			\begin{scope}[
				plane origin = {(0.1045010238,-0.3216210805,0.9291210788)},
				plane x = {(0.1408570595,0.6239964983,1.2523634542)},
				plane y = {(-0.8892332741,-0.3216210805,1.0408893444)},
				canvas is plane
			]
				\draw[fill = gray!20] (0,0) circle (0.14957753006384342);
			\end{scope}
			\begin{scope}[
				plane origin = {(0,0,0.98875)},
				plane x = {(0,1,0.98875)},
				plane y = {(-1,0,0.98875)},
				canvas is plane
			]
				\draw[fill = gray!20] (0,0) circle (0.14957753006384342);
			\end{scope}

			\begin{scope}[
				canvas is xz plane at y = 0
			]
				\shade[ball color = lightgray, opacity = 0.5] (0,0,0) circle (1);
			\end{scope}
		\end{tikzpicture}

		\caption{Configuration of the domain $\Omega$ in 3D. The accessible part $\Gamma_0$ is represented by the patches on the sphere.}
		\label{fig:configuration_3d}
	\end{figure}

	As in the 2D case, the physical constants are all fixed to 1, leading to a background refractive index equal to $\kappa_0 = 1 + i$. The amplitude of the perturbation is given by $a_\text{ex} \coloneqq 0.2$.

	Data are generated on a mesh of the unit ball $\Omega$ of size \scnum{0.0795939}, using \num{480175} tetrahedrons (\num{83619} vertices and \num{577278} edges). We use six incident waves of directions
	\[\bfeta \in \{(\pm 1,0,0), (0,\pm 1,0), (0,0,\pm 1)\}.\]
	The data transmission problem is solved on a mesh of $\mcV$ of size \scnum{0.0904978}, using \num{160156} tetrahedrons (\num{34520} vertices and \num{211511} edges). The cost function $j$ is minimized on a mesh of the unit ball of size \scnum{0.271142}, using \num{28020} tetrahedrons (\num{5605} vertices and \num{35653} edges). Results are summarized in \autoref{tab:3d_unit} and illustrated in \autoref{fig:reconstruction_3d}.

	\begin{table}[hbt]
		\centering

		\begin{tabular}{c|ccc}
			\hline
			Parameter & Exact value & Approximation & Relative error \\
			\hline
			Center & (\num{-0.4}, \num{0}, \num{0}) & (\rounded[5]{-0.35166079}, \rounded[5]{-0.03888007913}, \rounded[5]{-0.02515070436}) & \scnum[5]{0.1673487695350508} \\
			Radius & \num{0.2} & \rounded[5]{0.1932578736} & \scnum[5]{0.0337106320000001} \\
			Amplitude & \num{0.2} & \rounded[5]{0.2175963333} & \scnum[5]{0.08798166649999994} \\
			\hline
		\end{tabular}

		\caption{Reconstruction of a spherical perturbation in the unit ball in 3D, with unitary physical parameters.}
		\label{tab:3d_unit}
	\end{table}

	\begin{figure}[hbt]
		\centering
		\includegraphics[width=0.6\textwidth]{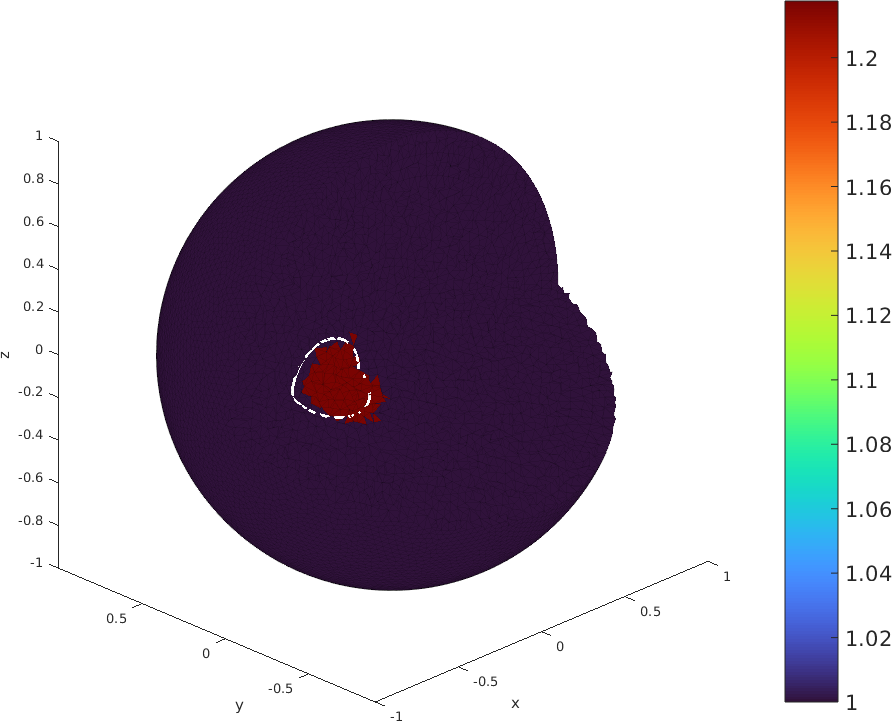}
		\caption{Real part of the reconstructed refractive index in the 3D unit ball. The boundary of the exact support of the perturbation is shown as a white line.}
		\label{fig:reconstruction_3d}
	\end{figure}

	\subsection{Microwave regime}

	Coming back to a two-dimensional setting, we now propose a slightly more complicated geometry. Here, the domain $\Omega$ represents a head profile. The known neighborhood $\mcV$ is the region between the boundary of $\Omega$ and an ellipse, and we choose a bigger ellipse to represent the artificial boundary $\Gamma_\text{int}$. The accessible part of the boundary is the union of 34 patches. The perturbation will be located in the disk centered at $(0.05,0.4)$ and of radius $0.1$. This configuration is illustrated in \autoref{fig:configuration_head}.

	\begin{figure}[!hbt]
		\centering

		\begin{tikzpicture}[scale = 2.5]
			\newcommand{\domainboundarycoordinates}{
				(-0.49407477804999766,0.445666948122729) .. controls (-0.49407477804999766,0.445666948122729) and (-0.5104511135366977,0.7430728902394812) .. (-0.3393819886902455,0.8321826420179219)
				.. controls (-0.18741085248829192,0.9113443027741972) and (-0.020740884965922607,0.8899706396788705) .. (0.12135063204877973,0.8316944873013291)
				.. controls (0.5071098293253042,0.673482617554397) and (0.687028594801552,0.23152937384998887) .. (0.687028594801552,0.23152937384998887)
				.. controls (0.8894773090345722,0.1951024878502941) and (0.9960739343326835,0.07171017891846673) .. (0.9998648043889664,-0.05275984545300609)
				.. controls (1.004245923233998,-0.19661638115194657) and (0.9017425449709044,-0.2996319778346144) .. (0.7626697023381308,-0.34493262395619745)
				.. controls (0.7409781523334679,-0.4442592400946668) and (0.6435049062338472,-0.5066110095904269) .. (0.6463473614121785,-0.5716151793615252)
				.. controls (0.6481233147620682,-0.6122435990596558) and (0.7335727058099801,-0.6708859506004162) .. (0.7177408720414453,-0.7250197981605773)
				.. controls (0.6856240109636759,-0.8348380586244989) and (0.5549670300304413,-0.8024907885204979) .. (0.5067265581363908,-0.7815142689483295)
				.. controls (0.384139286068613,-0.7914049343794536) and (0.3068791785462604,-0.7368080372024104) .. (0.25218279078888306,-0.7063983229049539)
				.. controls (0.25218279078888306,-0.7063983229049539) and (0.07569301000568392,-0.8665037717622859) .. (-0.17380946735625735,-0.8831075408559926)
				.. controls (-0.33259728900256236,-0.8936743703059854) and (-0.43551982944897,-0.7798089818050326) .. (-0.43551982944897,-0.7798089818050326)
				.. controls (-0.47275850279508097,-0.7780144320849869) and (-0.5252478733423545,-0.7681737908895823) .. (-0.5448051179442234,-0.6870880380987554)
				.. controls (-0.6475967399447334,-0.6637691197417483) and (-0.6681942654931241,-0.6023637195232224) .. (-0.6671171637336936,-0.5272085351388178)
				.. controls (-0.6671171637336936,-0.5272085351388178) and (-0.9127763329947863,-0.5299732574898973) .. (-0.979363425793556,-0.37897517635635186)
				.. controls (-1.0268377272371827,-0.2713185021702429) and (-0.9977710429783501,-0.08194257124757558) .. (-0.8356927050252134,-0.0304684695289642)
				.. controls (-0.7295148662637392,0.0032523284747527515) and (-0.5246390281834499,-0.07505434556723435) .. (-0.5246390281834499,-0.07505434556723435)
				.. controls (-0.5259668090125821,-0.02787669342033463) and (-0.5256246358017133,0.03619557504951054) .. (-0.5128507891416166,0.09092314891963928)
				.. controls (-0.5128507891416166,0.09092314891963928) and (-0.9761999972663337,0.1880649333176008) .. (-0.9511006624305995,0.32450999726912305)
				.. controls (-0.9263940829325583,0.4588199933666748) and (-0.49407477804999766,0.445666948122729) .. cycle
			}

			\draw[fill = gray, fill opacity = 0.1, even odd rule]
				\domainboundarycoordinates

				(0.05,0) ellipse (0.4 and 0.6)
			;

			\begin{scope}
				\clip
					(-0.6438406667748302,-0.6321525019820624) circle (0.075)
					(-0.5080016325688554,-0.7549658753737656) circle (0.075)
					(-0.36099774623636227,-0.8442846923859133) circle (0.075)
					(-0.1563087905835232,-0.8833616748287281) circle (0.075)
					(0.026050460816278598,-0.8405630750104072) circle (0.075)
					(0.18608000796304314,-0.7568266840615187) circle (0.075)
					(0.34983117248531403,-0.7586874927492717) circle (0.075)
					(0.5787106410789429,-0.8052077099430988) circle (0.075)
					(0.7071064405339054,-0.6823943365513955) circle (0.075)
					(0.6736118841543499,-0.49631346777608754) circle (0.075)
					(0.8113117270480781,-0.3288406858783105) circle (0.075)
					(0.9620372307560774,-0.1948624603600888) circle (0.075)
					(0.9918101697601265,0.013548112668256174) circle (0.075)
					(0.863414370305164,0.1698560424395148) circle (0.075)
					(0.6754726928421029,0.2610356681394157) circle (0.075)
					(0.574989023703437,0.4433949195392173) circle (0.075)
					(0.4317067547464495,0.6164501275002536) circle (0.075)
					(0.27539882497519086,0.7448459269552161) circle (0.075)
					(0.08373553013662383,0.8527728308448947) circle (0.075)
					(-0.10606695601419025,0.8825457698489438) circle (0.075)
					(-0.31819914641804115,0.8490512134693888) circle (0.075)
					(-0.44659494587300386,0.715072987951167) circle (0.075)
					(-0.48939354569132454,0.5382961626146243) circle (0.075)
					(-0.6029028756442625,0.4489773456024766) circle (0.075)
					(-0.8150350660481134,0.4136219805351682) circle (0.075)
					(-0.9434308655030759,0.28708698976795866) circle (0.075)
					(-0.7927053617950766,0.1754384685027739) circle (0.075)
					(-0.617789345146287,0.11961420787018162) circle (0.075)
					(-0.52288810207088,0.013548112668256174) circle (0.075)
					(-0.7927053617950766,-0.5167823633413715) circle (0.075)
					(-0.969482187131619,-0.4144378855149521) circle (0.075)
					(-0.9955335087601622,-0.22091378198863176) circle (0.075)
					(-0.8913282222459897,-0.05158019140310177) circle (0.075)
					(-0.6698919884033734,-0.034832913213324046) circle (0.075)
				;
				\draw[line width = 0.7mm]
					\domainboundarycoordinates
				;
			\end{scope}

			\draw[dashed]
				(0.05,0) ellipse (0.48 and 0.68)
			;

			\draw[dotted]
				(0.05,0.4) circle (0.1)
			;

			\node at (-0.05,0.75) {$\Gamma_\text{int}$};
			\node at (-0.05,-0.7) [node font = \LARGE] {$\mcV$};
			\node at (0.05,0.4) {$D$};
		\end{tikzpicture}

		\caption{Configuration of the domain $\Omega$ for the head profile. The thicker parts on the boundary is $\Gamma_0$. The artificial boundary is the dashed line, while the dotted line represents the boundary of the support of the perturbation.}
		\label{fig:configuration_head}
	\end{figure}

	We use here the common definitions of the physical constants. The permittivity in vacuum is then set to $\eps_0 \coloneqq \SI{8.854e-12}{\farad\per\metre}$ and the magnetic permeability of the vacuum is set to $\mu_0 \coloneqq \SI{4\pi e-7}{\henry\per\metre}$. We choose a frequency in the microwave regime, that is $\omega \coloneqq \SI{1e8}{\hertz}$. The background permittivity of the medium is fixed to $\eps \coloneqq \SI{1e-10}{\farad\per\metre}$ and its background conductivity is fixed to $\sigma \coloneqq \SI{0.33}{\siemens\per\metre}$. This yields $\kappa_0 \approx \num{1.1294e1} + \num{3.7271e2}i$. As previously, the exact refractive index is a perturbation of this background index:
	\[\kappa_\text{ex} \coloneqq \kappa_0(1 + a_\text{ex}\chi_{D_\text{ex}}),\]
	with $a_\text{ex} \coloneqq 0.1$ and $D_\text{ex}$ the disk centered at $\bfx_\text{ex} \coloneqq (0.05,0.4)$ and of radius $r_\text{ex} \coloneqq 0.1$.

	Data are generated on a mesh of $\Omega$ of size \scnum{0.00280451}, using \num{1396929} triangles (\num{700240} vertices and \num{2097168} edges). We use 16 incident waves, defined in the same way as in the unit disk case (see \autoref{sec:numerics_unit_disk}). A noise is added to the generated data, using the following procedure. Each degree of freedom of the tangential trace $\bfg_D$ of the field is perturbed by the addition of a random complex number in such a way that
	\[\frac{\norm{\bfg_D^\eta - \bfg_D}[0,\Gamma_0]}{\norm{\bfg_D}[0,\Gamma_0]} = \eta,\]
	where $\eta > 0$ is the level of noise and $\bfg_D^\eta$ is the noised trace. In this test, we fix this level of noise to $\eta = 0.02$. The data transmission problem is solved on a mesh of $\mcV$ of size \scnum{0.00345773}, using \num{622148} triangles (\num{313135} vertices and \num{935283} edges). The minimization is achieved on a mesh of $\Omega$ of size \scnum{0.0136741}, using \num{60194} triangles (\num{30474} vertices and \num{90667} edges). The retrieved parameters are summarized in \autoref{tab:2d_head}, with corresponding refractive index shown in \autoref{fig:reconstruction_2d_microwave}.

	\begin{table}[hbt]
		\centering

		\begin{tabular}{c|ccc}
			\hline
			Parameter & Exact value & Approximation & Relative error \\
			\hline
			Center & (\num{0.05}, \num{0.4}) & (\rounded[5]{0.05501903145}, \rounded[5]{0.3847967252}) & \scnum[5]{0.03971670769068229} \\
			Radius & \num{0.1} & \rounded[5]{0.1179580571} & \scnum[5]{0.17958057099999994} \\
			Amplitude & \num{0.1} & \rounded[5]{0.07054430332} & \scnum[5]{0.29455696680000004} \\
			\hline
		\end{tabular}

		\caption{Reconstruction of a spherical perturbation in a head profile in 2D, with realistic physical parameters in the microwave regime.}
		\label{tab:2d_head}
	\end{table}

	\begin{figure}[!hbt]
		\centering
		\includegraphics[width=0.5\textwidth]{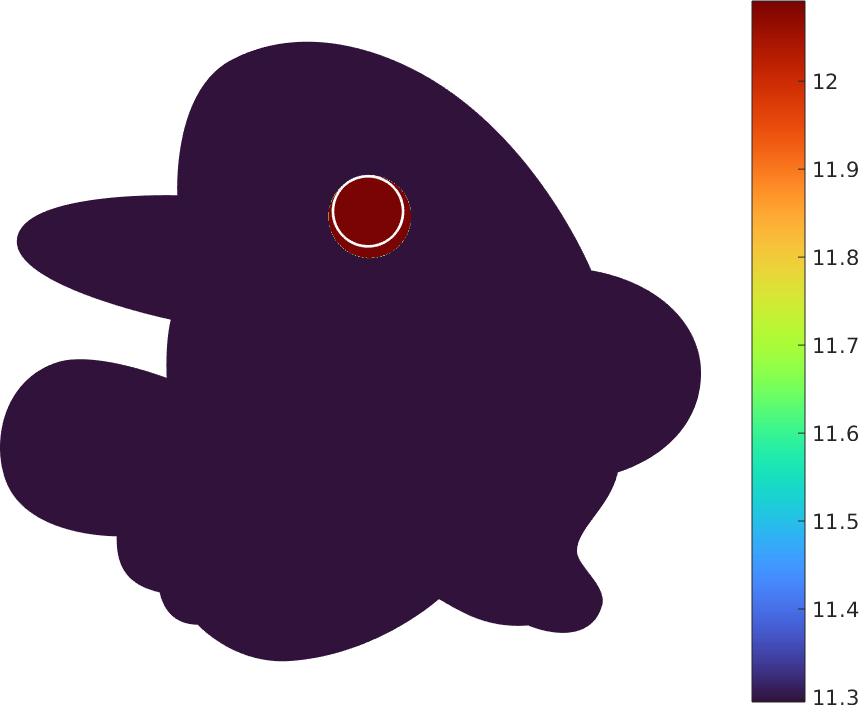}
		\caption{Real part of the reconstructed refractive index in the geometry of a 2D head profile in microwave regime. The boundary of the exact support of the perturbation is shown as a white line.}
		\label{fig:reconstruction_2d_microwave}
	\end{figure}

	\subsection{Discussion about the location}

	In all tested configurations, the minimization of the cost function defined in \autoref{sec:cost} leads to good results, with a satisfying precision. One can notice that the locations of the reconstructed perturbations are always deeper than the exact supports. This can actually be explained by the first step of the procedure, that is the transmission of the boundary data by the quasi-reversibility method. Indeed, in \autoref{tab:2d_unit_exact}, we see the parameters retrieved when we use the exact data as entry of the minimization step, instead of the transmitted data, in the case of the unit disk in $\R^2$ (see \autoref{sec:numerics_unit_disk}). The reconstructed perturbation is shown in \autoref{fig:reconstruction_2d_unit_exact}.

	\begin{table}[hbt]
		\centering

		\begin{tabular}{c|ccc}
			\hline
			Parameter & Exact value & Approximation & Relative error \\
			\hline
			Center & (\num{-0.4}, \num{0}) & (\rounded[5]{-0.3975437238}, \rounded[5]{0.002265724414}) & \scnum[5]{0.008354190524402283} \\
			Radius & \num{0.2} & \rounded[5]{0.2030149992} & \scnum[5]{0.015074995999999896} \\
			Amplitude & \num{0.1} & \rounded[5]{0.09822489383} & \scnum[5]{0.017751061699999987} \\
			\hline
		\end{tabular}

		\caption{Reconstruction of a spherical perturbation in the unit disk in 2D, with unitary physical parameters. The transmission step is skipped and exact data are used.}
		\label{tab:2d_unit_exact}
	\end{table}

	\begin{figure}[hbt]
		\centering
		\includegraphics[width=0.5\textwidth]{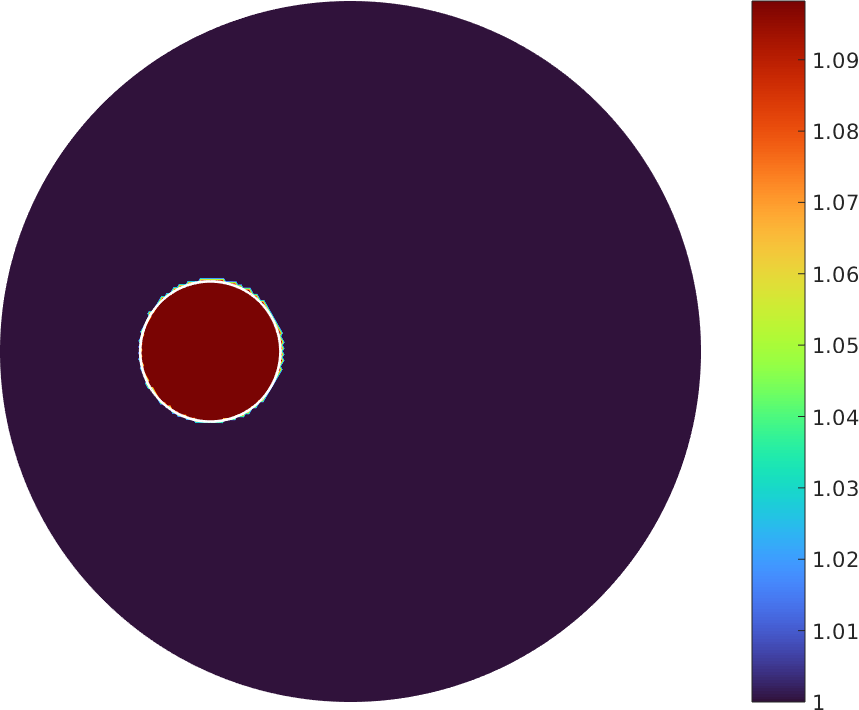}
		\caption{Real part of the reconstructed refractive index in the unit disk with exact data on the whole boundary. The boundary of the exact support of the perturbation is shown as a white line.}
		\label{fig:reconstruction_2d_unit_exact}
	\end{figure}

	To understand the reason behind this behavior, we compare in \autoref{fig:comparison_traces} the trace of the difference field $(\bfE[\kappa_\text{ex}] - \bfE[\kappa_0]) \times \bfn$ on $\Gamma_\text{int}$ with the trace resulting from the resolution of the transmission problem. It can be observed that the reconstructed surface peak is well-located. However, it has a lower amplitude and occupies more space over the surface. As it has been observed in \cite{DarbasHeleineLohrengel_sens}, the space occupied by the surface peak depends on the depth of the perturbation: the deeper a perturbation is, the larger will be the surface peak. As the reconstructed surface peak is larger than the exact one, it is then natural to reconstruct deeper perturbations at the end. It is not clear currently whether or not it is possible to predict the deformation of the peak due to the transmission step in order to correct it and obtain more precise results.

	\begin{figure}[hbt]
		\centering
		\includegraphics[width=0.7\textwidth]{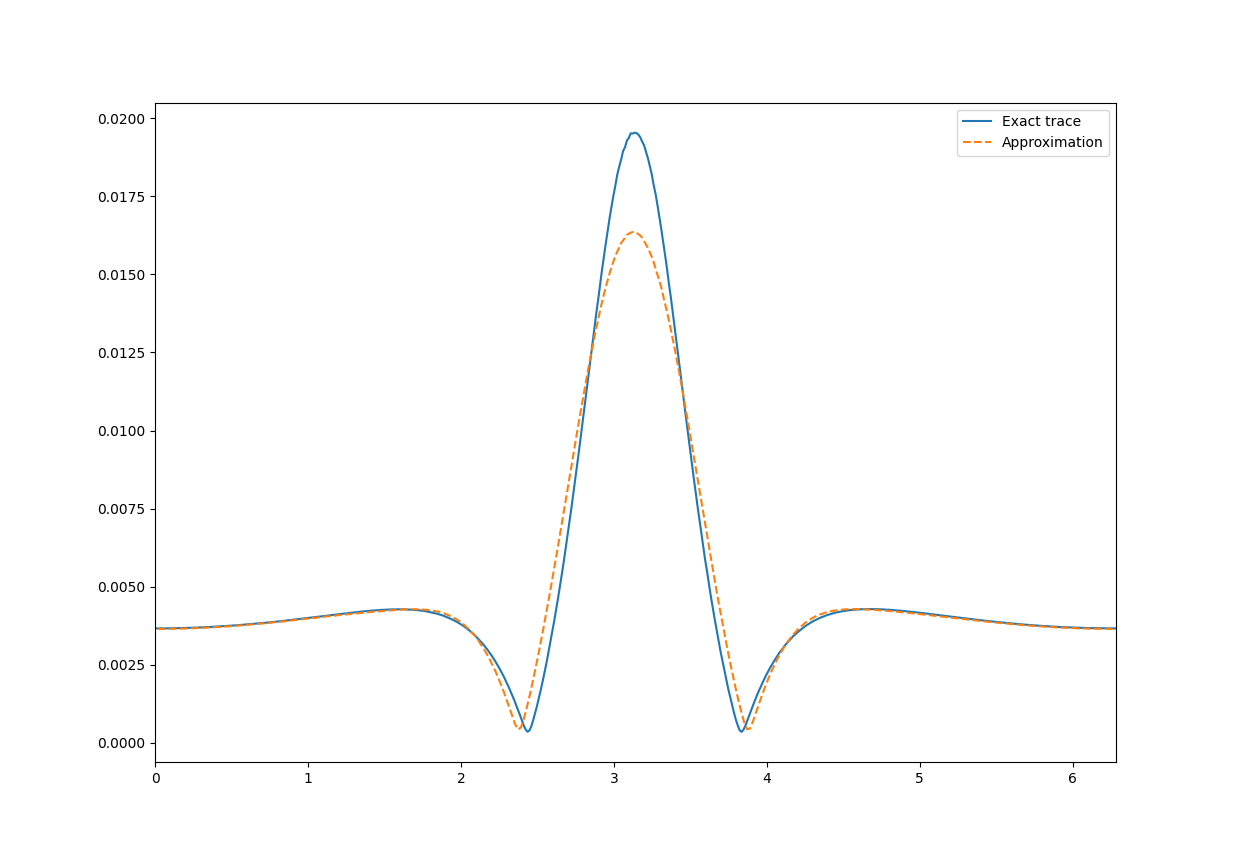}
		\caption{Comparison of the exact tangential trace of the difference field $(\bfE[\kappa_\text{ex}] - \bfE[\kappa_0]) \times \bfn$ on $\Gamma_\text{int}$ and its approximation by the quasi-reversibility method. The circle $\Gamma_\text{int}$ is unfolded: the $x$-axis shows the angle of a point while $y$-axis shows the value of the modulus of the trace at this point.}
		\label{fig:comparison_traces}
	\end{figure}

	\subsection{More complex perturbations}

	\subsubsection{Approximation by a ball}

	In the previous test cases, the algorithm is looking for a perturbation of constant amplitude contained in a single ball. One can wonder what it finds out with data generated from perturbations that do not satisfy these assumptions. In \autoref{fig:reconstruction_2d_complex}, we show the results of the complete inversion procedure in such cases, using the same global configuration as in the unit disk case shown in \autoref{sec:numerics_unit_disk}. In the first test, the amplitude of the perturbation is still constant and equal to \num{0.1}, but the support has a shape of a star instead of a ball. In the second test, we go back to a disk, but the amplitude is not constant: it is set to
	\[a_\text{max}e^{-\frac{r_c^2}{1 - r_c^2}},\]
	where $a_\text{max} \coloneqq 0.1$ is the maximum amplitude and $r_c$ is the distance of the current point to the center of the disk, divided by the radius of the disk. Then, $\kappa \in \mcC^\infty(\Omega)$, equal to $\kappa_0$ outside of the perturbation and to $\kappa_0(1 + a_\text{max})$ on the center of the perturbation.

	\begin{figure}[p]
		\centering

		\includegraphics[width=0.45\textwidth]{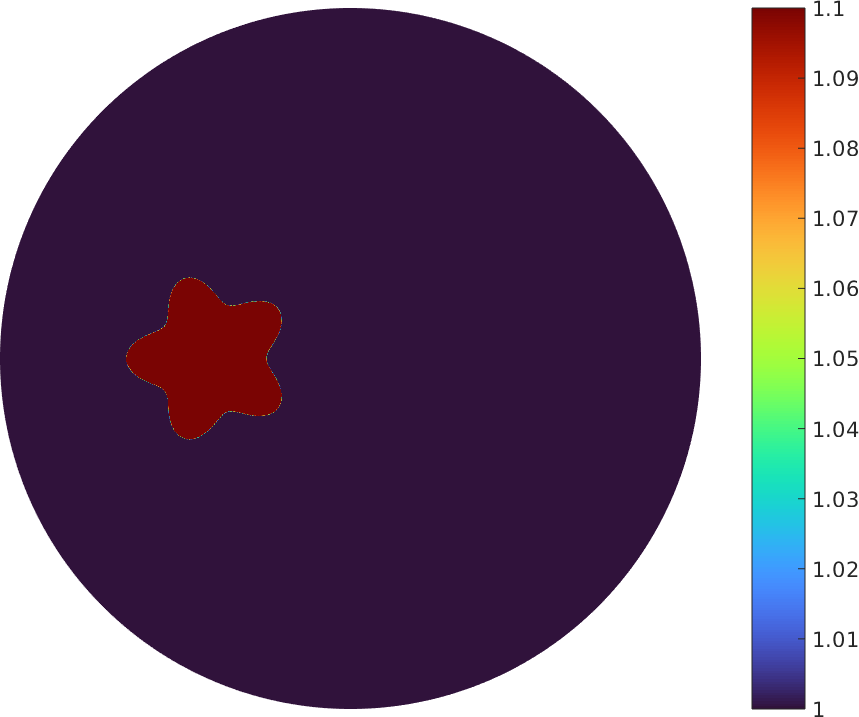}
		\hfill
		\includegraphics[width=0.45\textwidth]{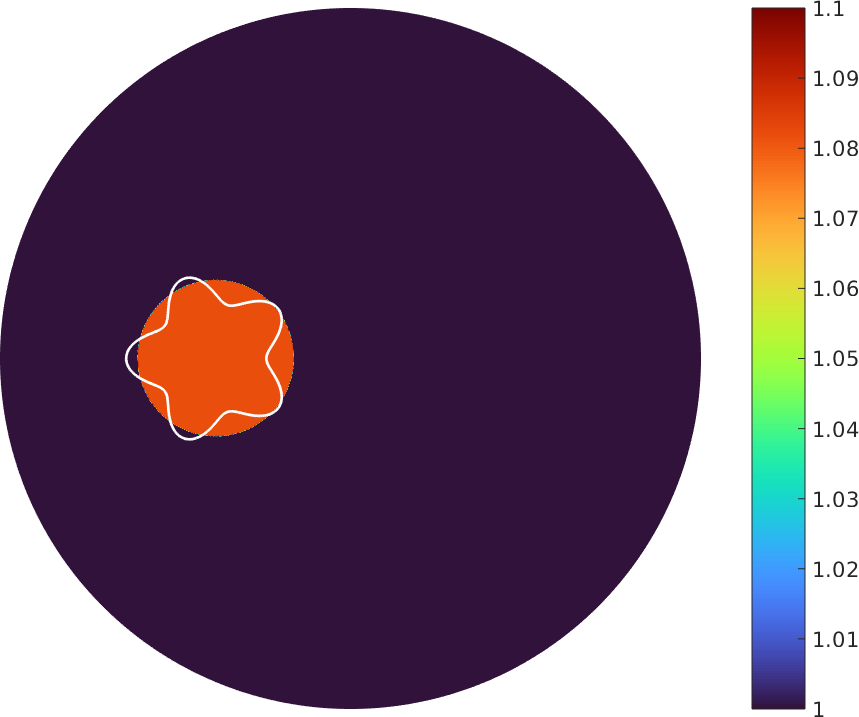}

		\includegraphics[width=0.45\textwidth]{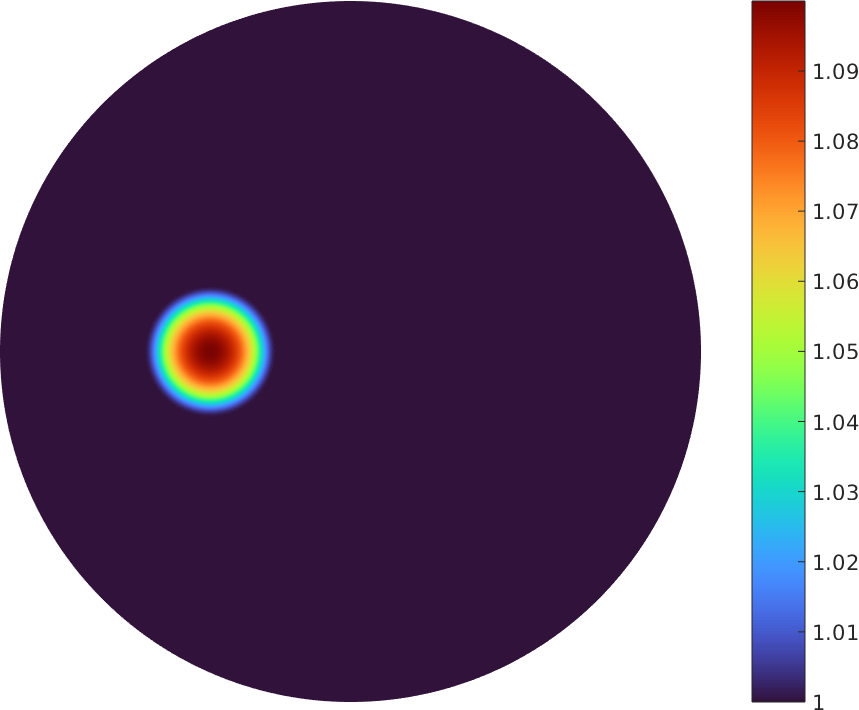}
		\hfill
		\includegraphics[width=0.45\textwidth]{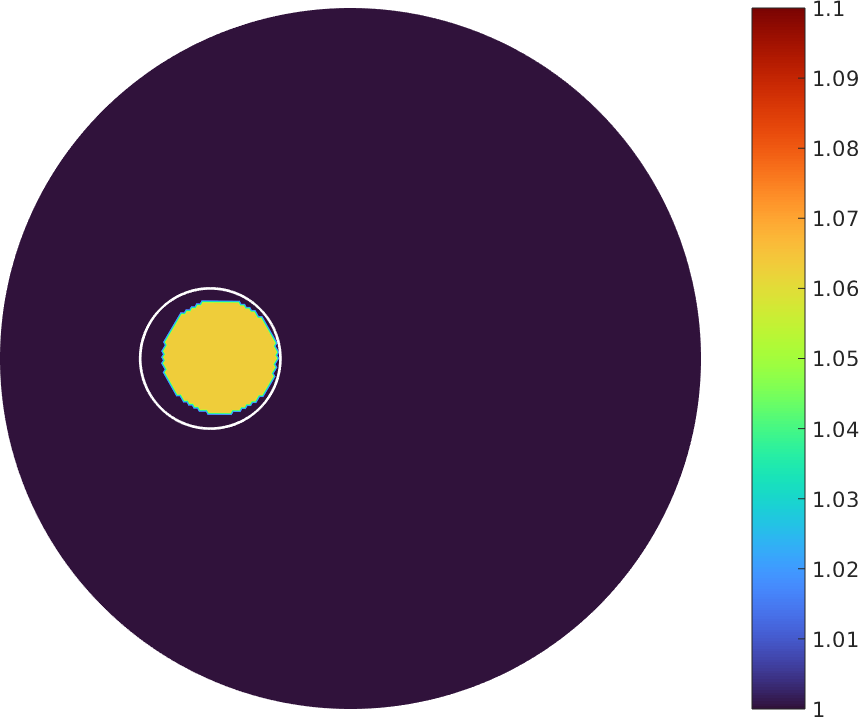}

		\caption{Reconstruction of more complex perturbations in the unit disk in 2D. Top pictures: star-shaped perturbation with constant amplitude. Bottom pictures: disk-shaped perturbation with variable amplitude. Left pictures: real part of the exact refractive index. Right pictures: results of the complete inversion procedures.}
		\label{fig:reconstruction_2d_complex}
	\end{figure}

	In both cases, the reconstructed perturbation is deeper than the exact support, following the same behavior discussed in the previous section. In the case of the star-shaped perturbation, the algorithm finds an amplitude equal to \rounded[5]{0.090844674} instead of \num{0.1}. As the procedure is looking for a constant amplitude, it does not retrieve the variations of the second case: it finds an amplitude equal to \rounded[5]{0.06311042633}. It is interesting that, in both cases, the procedure provides a rough idea of the perturbation, which could serve as a good initial guess for an algorithm looking for more parameters.

	\subsubsection{Ellipsoids}

	As a first step to handle more complex shapes, we propose to reconstruct ellipsoids, in 2D and 3D configurations. Except for the shape of the perturbation's support, the settings are the ones described in \autoref{sec:numerics_unit_disk} and \autoref{sec:numerics_unit_ball}. The results of the complete inversion procedure are listed in \autoref{tab:2d_ellipse} for the 2D case, and in \autoref{tab:3d_ellipsoid} for the 3D case. They are illustrated in \autoref{fig:reconstruction_ellipsoid}.

	\begin{table}[hbt]
		\centering

		\begin{tabular}{c|ccc}
			\hline
			Parameter & Exact value & Approximation & Relative error \\
			\hline
			Center & (\num{-0.4}, \num{0}) & (\rounded[5]{-0.3918151755}, \rounded[5]{0.001865934299}) & \scnum[5]{0.020987058944738006} \\
			$x$-radius & \num{0.15} & \rounded[5]{0.1396220837} & \scnum[5]{0.06918610866666657} \\
			$y$-radius & \num{0.25} & \rounded[5]{0.2639521543} & \scnum[5]{0.05580861720000008} \\
			Amplitude & \num{0.1} & \rounded[5]{0.100963128} & \scnum[5]{0.009631279999999937} \\
			\hline
		\end{tabular}

		\caption{Reconstruction of an ellipsoidal perturbation in the unit disk in 2D, with unitary physical parameters.}
		\label{tab:2d_ellipse}
	\end{table}

	\begin{table}[hbt]
		\centering

		\begin{tabular}{c|ccc}
			\hline
			Parameter & Exact value & Approximation & Relative error \\
			\hline
			Center & (\num{-0.4}, \num{0}, \num{0}) & (\rounded[5]{-0.3524504152}, \rounded[5]{-0.04525513963}, \rounded[5]{-0.02454621712}) & \scnum[5]{0.17520536971803571} \\
			$x$-radius & \num{0.2} & \rounded[5]{0.1936337415} & \scnum[5]{0.0318312925} \\
			$y$-radius & \num{0.4} & \rounded[5]{0.3473021405} & \scnum[5]{0.13174464875} \\
			$z$-radius & \num{0.3} & \rounded[5]{0.2810993029} & \scnum[5]{0.06300232366666658} \\
			Amplitude & \num{0.2} & \rounded[5]{0.247375303} & \scnum[5]{0.2368765149999999} \\
			\hline
		\end{tabular}

		\caption{Reconstruction of an ellipsoidal perturbation in the unit ball in 3D, with unitary physical parameters.}
		\label{tab:3d_ellipsoid}
	\end{table}

	\begin{figure}[p]
		\centering

		\includegraphics[width=0.45\textwidth]{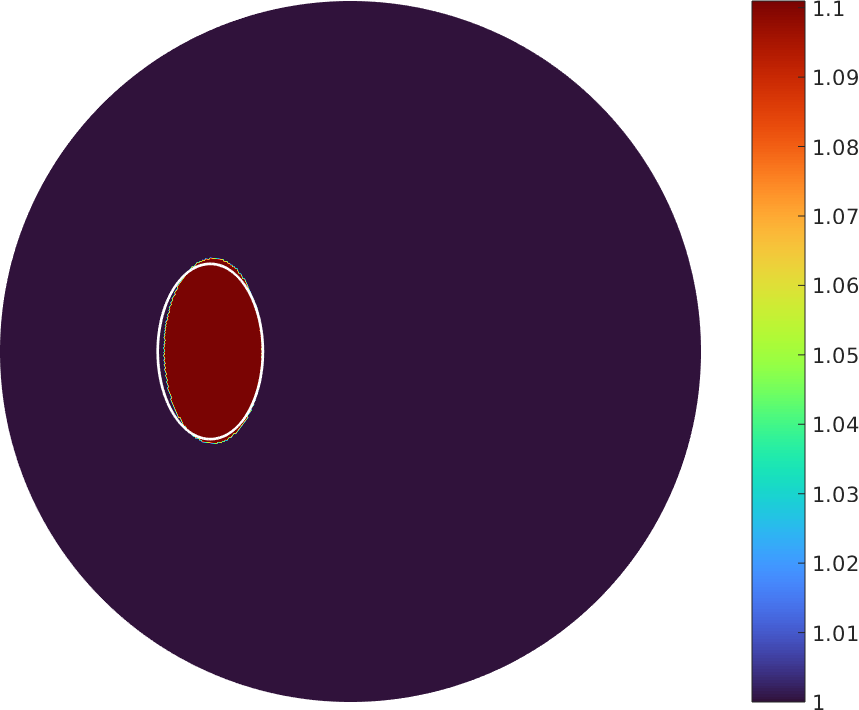}
		\hfill
		\includegraphics[width=0.45\textwidth]{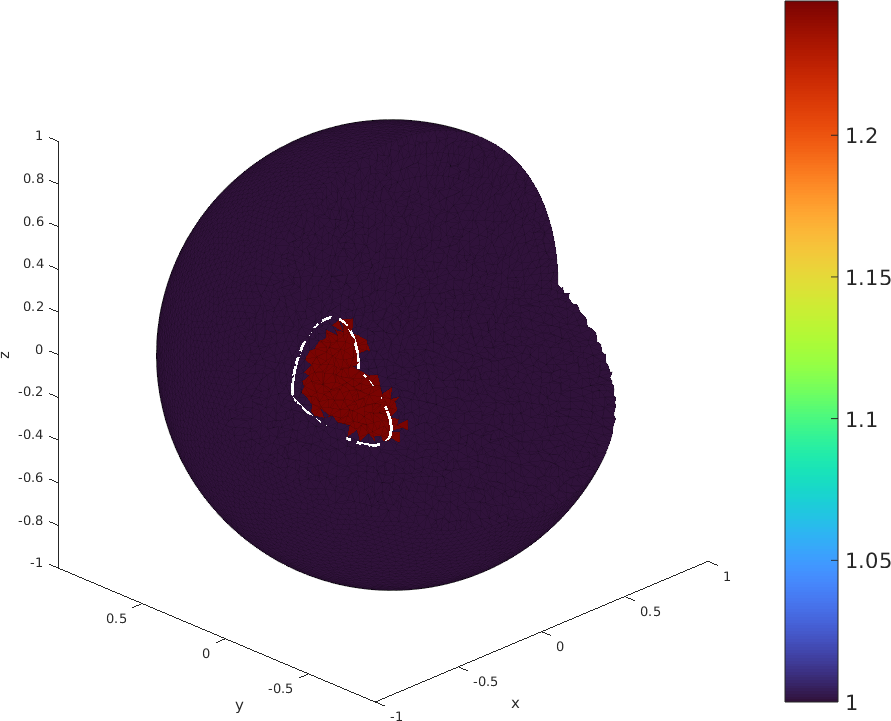}

		\caption{Reconstruction of an ellipsoidal perturbation in the unit disk (left picture) and in the unit ball (right picture).}
		\label{fig:reconstruction_ellipsoid}
	\end{figure}

	\subsubsection{Multiple parts}

	In the previous test cases, the support of the perturbation is always composed of a single connected component. When it is not the case, as it has been proved in \cite{DarbasHeleineLohrengel_sens}, each component will generate its own surface peak. This behavior allows us to automatically detect when there are several connected components, without having to a priori know their numbers. Then, in the localization step, each peak leads to an approximation of the projection of the corresponding component's center on $\Gamma_\text{int}$. The idea is then to minimize
	\[\tilde{j} \colon (d_1, \ldots, d_P, r_1, \ldots, r_P) \mapsto \min_a J\left(\bigcup_{i=1}^P D(d_i, r_i), a\right),\]
	where $P \in \N^*$ is the number of peaks that have been found. Once again, we consider a test case in the unit disk and another one in the unit ball. The results of the 2D case are listed in \autoref{tab:2d_multiple} and the ones for the 3D case are listed in \autoref{tab:3d_multiple}. Both are illustrated in \autoref{fig:reconstruction_multiple}.

	\begin{table}[hbt]
		\centering

		\begin{tabular}{c|ccc}
			\hline
			Parameter & Exact value & Approximation & Relative error \\
			\hline
			Center 1 & (\num{-0.55}, \num{-0.45}) & (\rounded[5]{-0.5315738734}, \rounded[5]{-0.4399620166}) & \scnum[5]{0.029527072733237108} \\
			Radius 1 & \num{0.1} & \rounded[5]{0.1218877446} & \scnum[5]{0.21887744599999995} \\
			Center 2 & (\num{0.4}, \num{0.6}) & (\rounded[5]{0.3848769869}, \rounded[5]{0.5855851716}) & \scnum[5]{0.02897256006267024} \\
			Radius 2 & \num{0.07} & \rounded[5]{0.08498304818} & \scnum[5]{0.2140435454285714} \\
			Amplitude & \num{0.1} & \rounded[5]{0.06953012227} & \scnum[5]{0.3046987773000001} \\
			\hline
		\end{tabular}

		\caption{Reconstruction of a perturbation having two connected components in the unit disk in 2D, with unitary physical parameters.}
		\label{tab:2d_multiple}
	\end{table}

	\begin{table}[hbt]
		\centering

		\begin{tabular}{c|ccc}
			\hline
			Parameter & Exact value & Approximation & Relative error \\
			\hline
			Center 1 & (\num{-0.5}, \num{0}, \num{0}) & (\rounded[5]{-0.4296754798}, \rounded[5]{-0.03324893377}, \rounded[5]{-0.02997686217}) & \scnum[5]{0.16672902571251635} \\
			Radius 1 & \num{0.2} & \rounded[5]{0.1877256965} & \scnum[5]{0.0613715175} \\
			Center 2 & (\num{0}, \num{0}, \num{0.6}) & (\rounded[5]{-0.0336711867}, \rounded[5]{0.01040770127}, \rounded[5]{0.5009860535}) & \scnum[5]{0.17516524468988853} \\
			Radius 2 & \num{0.1} & \rounded[5]{0.09801102098} & \scnum[5]{0.019889790200000035} \\
			Amplitude & \num{0.2} & \rounded[5]{0.2461683003} & \scnum[5]{0.23084150149999996} \\
			\hline
		\end{tabular}

		\caption{Reconstruction of a perturbation having two connected components in the unit ball in 3D, with unitary physical parameters.}
		\label{tab:3d_multiple}
	\end{table}

	\begin{figure}[!hbt]
		\centering

		\includegraphics[width=0.45\textwidth]{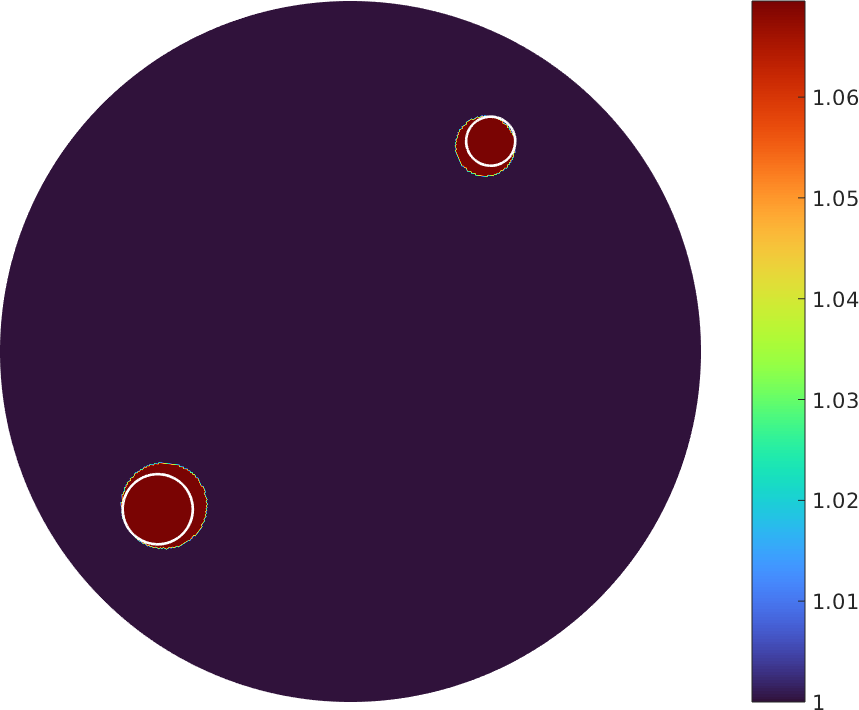}
		\hfill
		\includegraphics[width=0.45\textwidth]{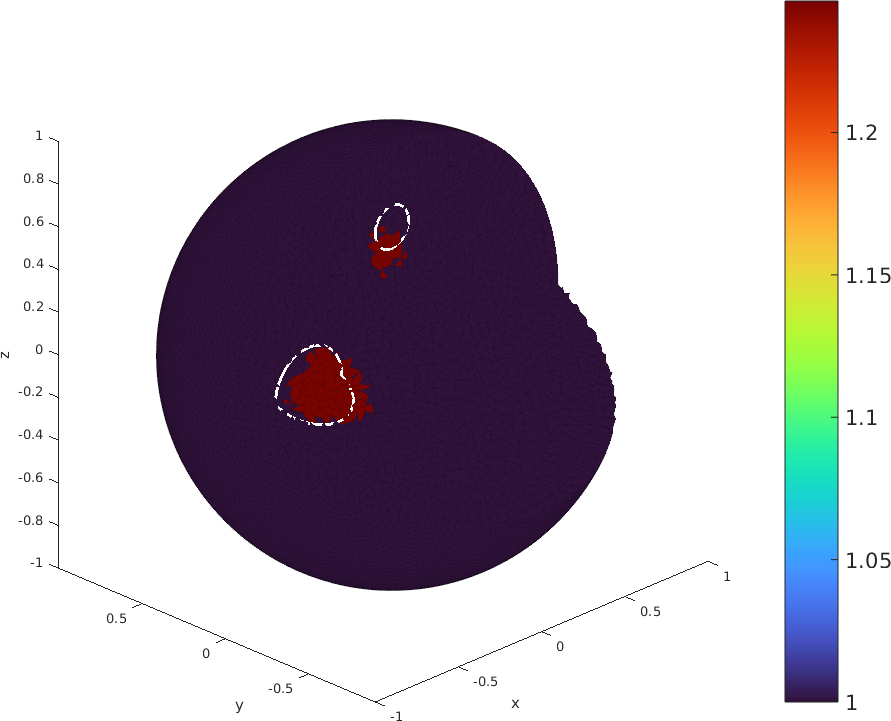}

		\caption{Reconstruction of a perturbation composed of two connected components in 2D (left picture) and in 3D (right picture).}
		\label{fig:reconstruction_multiple}
	\end{figure}

	\subsubsection{Other shapes}

	We propose here a simple idea to explore the case of more complex shapes. The first step is to apply the inverse procedure to get a ball whose center is denoted $(x_0, y_0)$ and radius is denoted $r_0$. Then, we slightly deform the boundary of this ball. To this end, in 2D, we parametrize this boundary:
	\[x(\theta) = x_0 + (r_0 + r(\theta)) \cos\theta, \quad y(\theta) = y_0 + (r_0 + r(\theta)) \sin\theta, \quad \text{with } \theta \in \interval[open right]{0}{2\pi},\]
	where $r$ describes the local perturbation of the radius. The goal is then to reconstruct this function. For that, we chose to approximate its decomposition into a Fourier series:
	\[r \approx \sum_{n=1}^{N_r} (a_n \cos(n \cdot) + b_n \sin(n \cdot)),\]
	with $N_r \in \N^*$ the number of pairs of coefficients $((a_n, b_n))_{n=1}^{N_r}$ to consider in the truncation. We obtain the following cost function to minimize:
	\[\tilde{j}_{N_r} \colon (d, r_0, a_1, \ldots, a_{N_r}, b_1, \ldots, b_{N_r}) \mapsto \min_a J(D(d, r_0, a_1, \ldots, a_{N_r}, b_1, \ldots, b_{N_r}), a),\]
	where $D(d, r_0, a_1, \ldots, a_{N_r}, b_1, \ldots, b_{N_r})$ is the currently tested support, with boundary described as above. The idea is to minimize $\tilde{j}_{N_r}$ using the previously found disk as an initial guess, with $N_r$ incrementally increased until no significant change is observed in the value of the cost function.

	The test case is the star-shaped perturbation shown in \autoref{fig:reconstruction_2d_complex}. Two configurations are considered: one with total data, where the data completion step is skipped, and one with partial data, using the complete inversion procedure. We obtain the reconstructions shown in \autoref{fig:reconstruction_2d_star}.

	\begin{figure}[!hbt]
		\centering

		\includegraphics[width=0.45\textwidth]{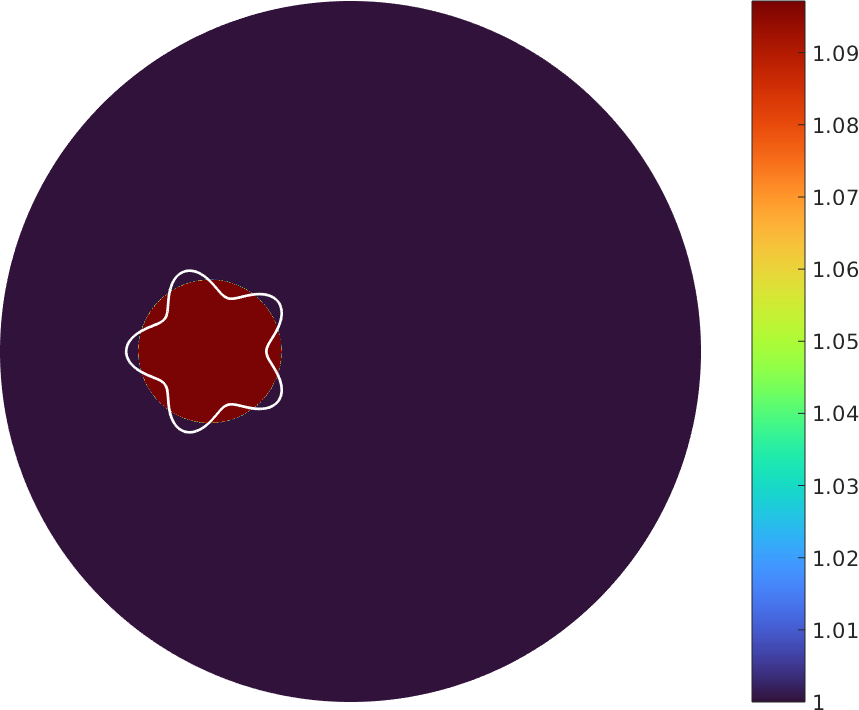}
		\hfill
		\includegraphics[width=0.45\textwidth]{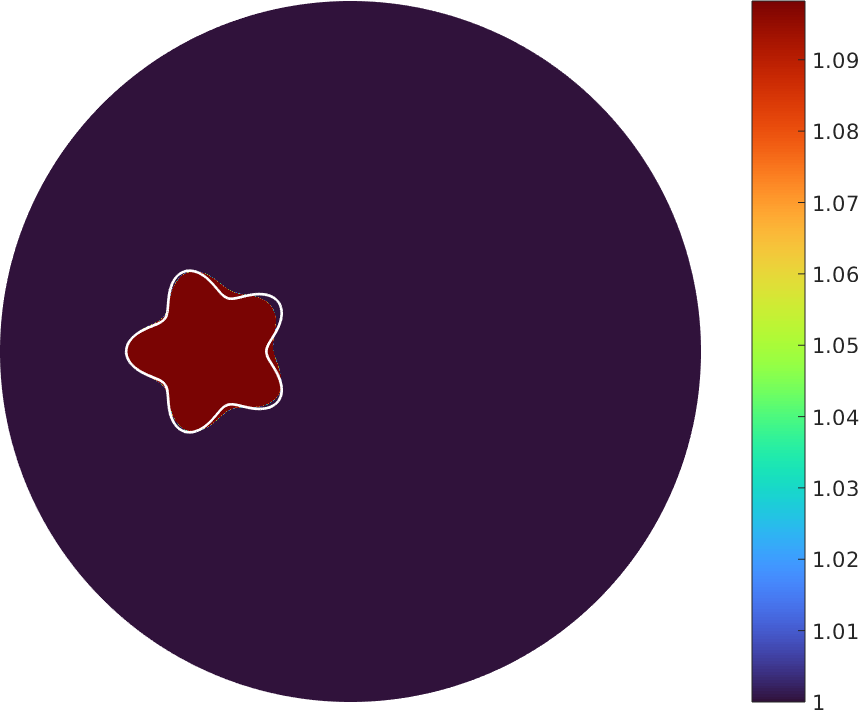}

		\includegraphics[width=0.45\textwidth]{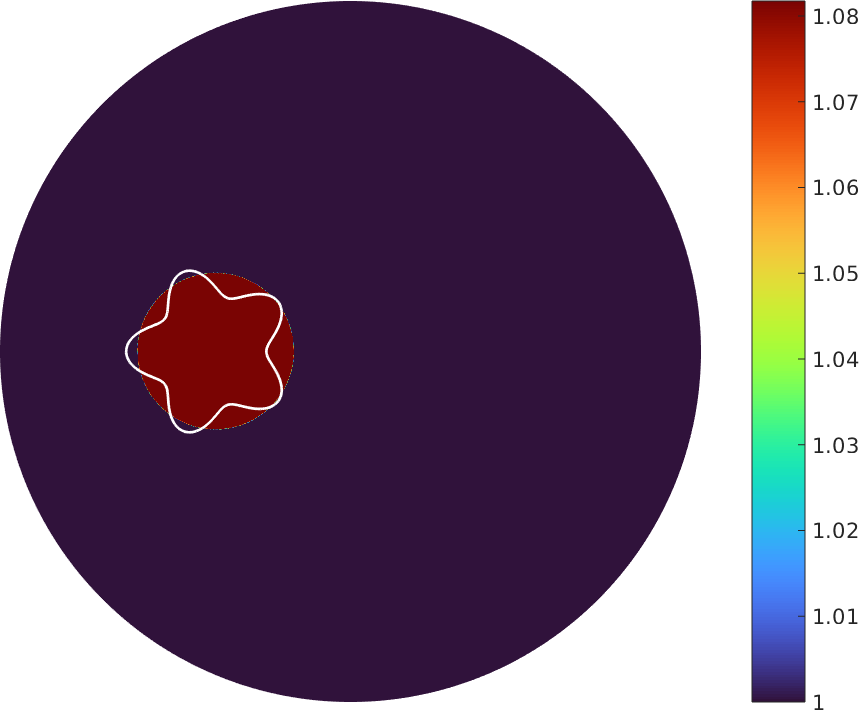}
		\hfill
		\includegraphics[width=0.45\textwidth]{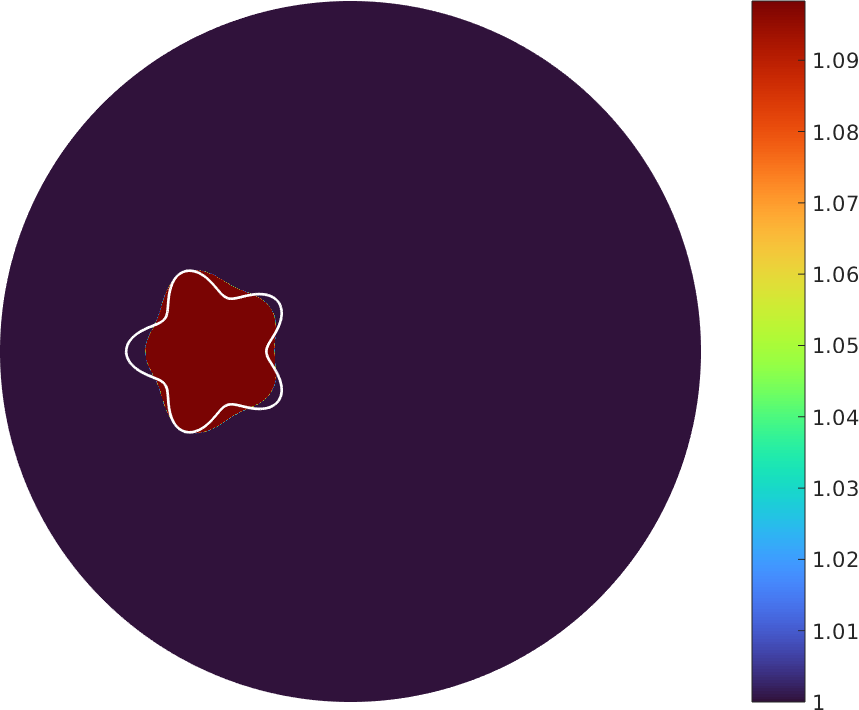}

		\caption{Reconstruction of a star-shaped perturbation, using total data on top figures and partial data on bottom figure. Left: approximation by a disk. Right: final result.}
		\label{fig:reconstruction_2d_star}
	\end{figure}

	While the reconstructed perturbation is close to the exact one in the case of total data, the case of partial data is less satisfying: the location is reasonable, but the shape is not fully recovered. When used in the direct solver, this perturbation leads to a tangential trace that is very close to the one we provide in the minimization, as illustrated in \autoref{fig:comparison_traces_min}. Hence, it seems that the issue comes from the data completion step that certainly deforms the shape of the perturbation.

	\begin{figure}[hbt]
		\centering

		\includegraphics[width=0.7\textwidth]{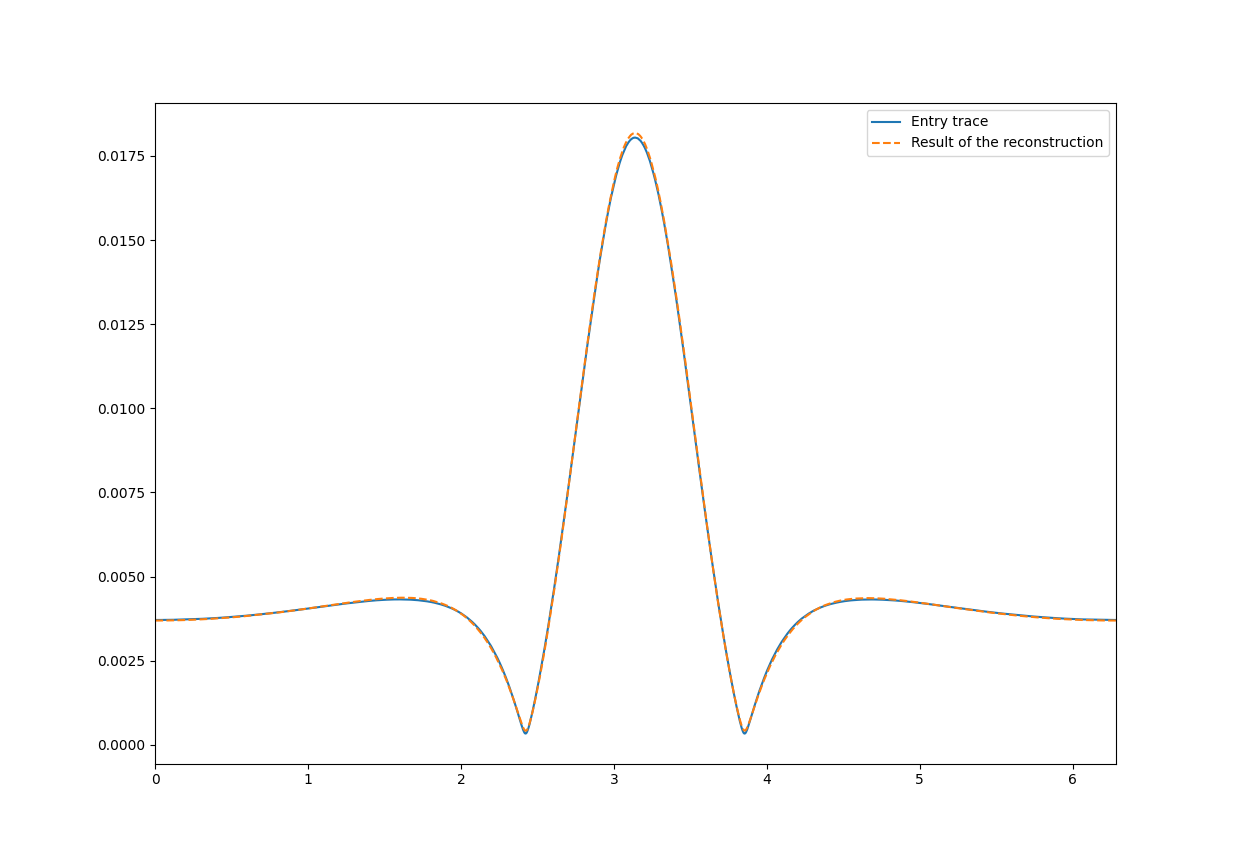}

		\caption{Comparison of a tangential trace of the difference field generated by the star-shaped perturbation, after data completion, with the trace generated from the reconstructed perturbation.}
		\label{fig:comparison_traces_min}
	\end{figure}

	It has to be noted that the initial minimization, where the goal is to approximate the support by a ball, should not be ommitted. As an example, we provide \autoref{fig:reconstruction_2d_star_fail} where the cost function is minimized with the exact same setting as for the partial data case in \autoref{fig:reconstruction_2d_star}, except that it is not initialized with the disk approximation.

	\begin{figure}[hbt]
		\centering

		\includegraphics[width=0.5\textwidth]{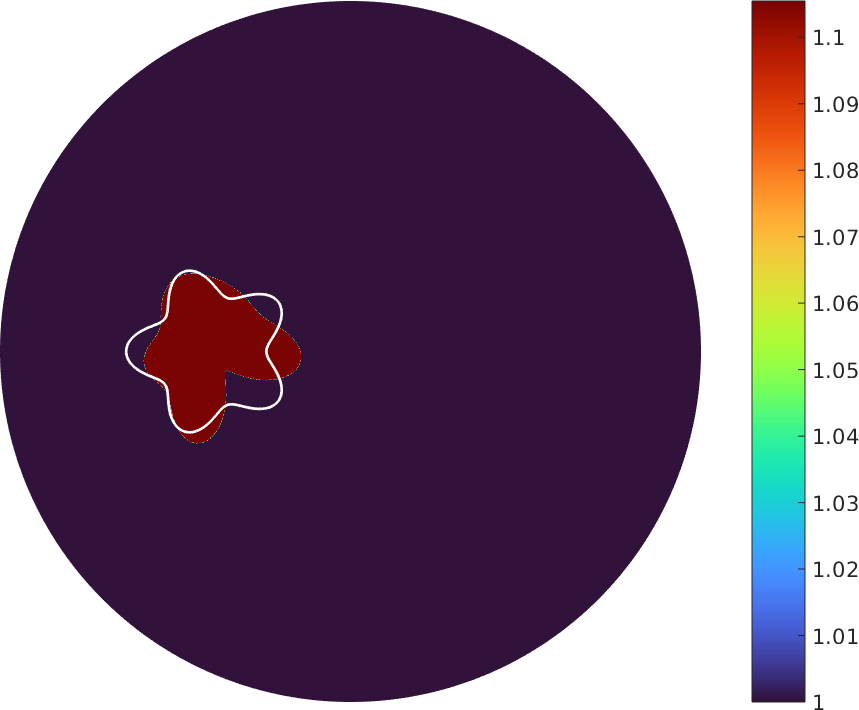}

		\caption{Result of the minimization of the cost function without using the previous rough approximation as initial guess.}
		\label{fig:reconstruction_2d_star_fail}
	\end{figure}

	\section{Concluding remarks}

	Using the iterated sensitivity equation, we were able to regularize the classical cost functional used to solve the inverse problem we are interested in. The complete inversion procedure provides good results when it comes to the reconstruction of a perturbation from the knowledge of partial surface measurements, even in cases where the shape or amplitude of the perturbation is not what the algorithm is expecting.

	In addition to the regularization it provides, the use of an asymptotic expansion is also interesting from a numerical point of view: using a direct solver, the evaluation of this cost functional is faster than the evaluation of the classical one, as it is always the same linear system that needs to be inverted. The same remark applies to the use of the iterated quasi-reversibility: all iterations, for all incident waves, share the same linear system. The limitations of using direct solvers are however well known. In particular, it prevents the use of fine meshes, especially in 3D.

	The examples of reconstruction of shape parameters led to promising results, especially in the case of total data. If one would like to capture more complex geometries in this case, it would be better to study other approaches, for instance involving shape derivatives. The case of partial data is more delicate: as observed in \autoref{fig:comparison_traces}, some informations are lost. A study of the (iterated) quasi-reversibility method is needed to better understand what happens exactly, and to determine whether it is possible or not to counter this effect.

	\section*{Acknowledgements}

	The author wants to thank M. Darbas, S. Lohrengel and J. Dardé for the kind discussions and valuable comments along this project.

\end{document}